\documentclass[final,onefignum,onetabnum]{siamltex1213}

\usepackage{dgleich-setup}
\usepackage{dgleich-math}
\renewcommand{\mat}{\boldsymbol}
\usepackage{epstopdf}
\usepackage{subfigure}

\let\oldthebibliography=\thebibliography
\let\oldendthebibliography=\endthebibliography

\usepackage[square]{natbib}
\renewcommand{\cite}{\citep}

\let\thebibliography=\oldthebibliography
\let\endthebibliography=\oldendthebibliography

  \providecolor{theblue}{RGB}{0,0,180}%
  \RequirePackage[colorlinks,pdfdisplaydoctitle
      ,citecolor=theblue
      ,linkcolor=theblue
      ,urlcolor=theblue
      ,hyperfootnotes=false,pagebackref]{hyperref}%
\renewcommand*{\backref}[1]{}%
\renewcommand*{\backrefalt}[4]{%
  \ifcase #1 %
    No citations.%
  \or
    Cited on page #2.%
  \else
    Cited on pages #2.%
  \fi
}%

\usepackage{ushort}
\newcommand{\cel}[1]{\ushort{#1}}

\newcommand{\celm}[1]{\cel{\mat{#1}}}

\providecommand{\cmP}{\ensuremath{\celm{P}}}

\providecommand{\cmR}{\ensuremath{\celm{R}}}

\newtheorem{remark}[theorem]{Remark}
\newtheorem{observation}[theorem]{Observation}

\newcommand{\fullkron}[2]{\underbrace{{#1} \kron \cdots \kron {#1}}_{\text{${#2}$ terms}}}
\let\p\relax
\DeclareMathOperator{\p}{Pr}
\newcommand{\given}{\mid}
\usepackage{todonotes}
\usetikzlibrary{arrows,backgrounds,patterns,matrix,shapes,fit,calc,shadows,plotmarks}

\makeatletter
\global\let\tikz@ensure@dollar@catcode=\relax
\makeatother

\newcommand{\occ}{\vw}
\newcommand{\occmat}{\mM}
\newcommand{\prtel}{\vv}

\newcommand{\occseq}[1]{\occ(#1)}
\newcommand{\xseq}[1]{X(#1)}
\newcommand{\yseq}[1]{Y(#1)}

\newcommand{\pijk}{\cmP_{ijk}}
\newcommand{\paak}[1]{\cmP_{\bullet\bullet#1}}

\graphicspath{{./FIG/}}

\title{The Spacey Random Walk: \\ a Stochastic Process for Higher-order Data\thanks{ARB is supported by a Stanford Graduate Fellowship. 
DFG is supported by NSF CCF-1149756, IIS-1422918, IIS-1546488, and the DARPA SIMPLEX program. 
DFG is also partially supported by an Alfred P. Sloan Research Fellowship.
LHL is supported by AFOSR FA9550-13-1-0133, DARPA D15AP00109, NSF IIS-1546413, DMS-1209136, and DMS-1057064.}
}

\author{Austin~R.~Benson\thanks{Institute for Computational and Mathematical Engineering, Stanford University, 475 Via Ortega, Stanford, CA 94305, USA (\email{arbenson@stanford.edu}).}
\and
David~F.~Gleich\thanks{Department of Computer Science, Purdue University, 305 North University Avenue, West Lafayette, IN 47907, USA (\email{dgleich@purdue.edu}).}
\and
Lek-Heng~Lim\thanks{Computational and Applied Mathematics Initiative, Department of Statistics, University of Chicago, 5747 South Ellis Avenue, Chicago, IL 60637, USA (\email{lekheng@uchicago.edu}).}}

\begin{document}

\maketitle

\begin{abstract}%
Random walks are a fundamental model in applied mathematics and are a common example of a Markov chain. 
The limiting stationary distribution of the Markov
chain represents the fraction of the
time spent in each state during the stochastic process.  
A standard way to compute this distribution for a random walk on a finite set of states
is to compute the Perron vector of the associated transition matrix.  
There are algebraic analogues of this Perron vector in terms of transition probability 
tensors of higher-order Markov chains. 
These vectors are nonnegative, have dimension equal to the dimension of the state space, and sum
to one and are derived by making an algebraic substitution in the 
equation for the joint-stationary distribution of a higher-order Markov chains. 
Here, we present the \emph{spacey random walk}, a non-Markovian
stochastic process whose stationary distribution is given by the tensor
eigenvector.  The process itself is a vertex-reinforced random walk, and its discrete
dynamics are related to a continuous dynamical system.  We analyze the
convergence properties of these dynamics and discuss numerical methods for
computing the stationary distribution.  Finally, we provide several applications
of the spacey random walk model in population genetics, ranking, and clustering data, 
and we use the process to analyze taxi trajectory
data in New York. This example shows definite non-Markovian structure. 
\end{abstract}

\section{Random walks, higher-order Markov chains, and stationary distributions}

Random walks and Markov chains are one of the most well-known and studied
stochastic processes as well as a common tool in applied mathematics. A random
walk on a finite set of states is a process that moves from state to state in a
manner that \emph{depends only on the last state} and an associated set of
transition probabilities from that state to the other states.  Here is one such
example, with a few sample trajectories of transitions:
\begin{center}
  \scalebox{0.7}{%
\begin{tikzpicture}
\begin{scope}[every node/.style={ultra thick, draw=black, ellipse, minimum width=45pt,
    align=center}]
    \node (1) at (5,10) {1};
    \node (2) at (10,10) {2};
    \node (3) at (7.5,7.5) {3};    
\end{scope}

\begin{scope}[every node/.style={fill=white,circle},
              every edge/.style={draw=black,very thick,>=triangle 60}]
    \path [->] (1) edge[out=140, in=80, loop] node {\footnotesize $1/2$} (1);    
    \path [->] (1) edge node {\footnotesize $1/4$} (2);
    \path [->] (1) edge[bend left=20] node {\footnotesize $1/4$} (3);

    \path [->] (2) edge[out=40, in=100, loop] node {\footnotesize $2/3$} (2);
    \path [->] (2) edge[bend right=20] node {\footnotesize $1/3$} (3);
    
    \path [->] (3) edge[bend left=20] node {\footnotesize $3/5$} (1);
    \path [->] (3) edge[bend right=20] node {\footnotesize $1/5$} (2);
    \path [->] (3) edge[out=-57.5, in=-122.5, loop] node {\footnotesize $1/5$} (3);    
\end{scope}

\newcommand{\aypos}{9.5}
\newcommand{\bypos}{\aypos-0.75}
\newcommand{\cypos}{\bypos-0.75}
\newcommand{\xshift}{7}
\newcommand{\cdotspace}{\phantom{.}}

\begin{scope}[every node/.style={ultra thick, align=center}]
    \node (a1) at (4.5+\xshift,\aypos) {1};
    \node (a2) at (6.0+\xshift,\aypos) {1};
    \node (a3) at (7.5+\xshift,\aypos) {3};    
    \node (a4) at (9.0+\xshift,\aypos) {3};
    \node (a5) at (10.5+\xshift,\aypos) {1 \cdotspace\large$\cdots$};
\end{scope}
\begin{scope}[every node/.style={above,},
              every edge/.style={draw=black,>=triangle 60}]
    \path [->] (a1) edge node {\scriptsize $1/2$} (a2);
    \path [->] (a2) edge node {\scriptsize $1/4$} (a3);
    \path [->] (a3) edge node {\scriptsize $1/5$} (a4);    
    \path [->] (a4) edge node {\scriptsize $3/5$} (a5);    
\end{scope}

\begin{scope}[every node/.style={ultra thick, align=center}]
    \node (b1) at (4.5+\xshift,\bypos) {2};
    \node (b2) at (6.0+\xshift,\bypos) {3};
    \node (b3) at (7.5+\xshift,\bypos) {3 \cdotspace\large$\cdots$};   
\end{scope}
\begin{scope}[every node/.style={above,},
              every edge/.style={draw=black,>=triangle 60}]
    \path [->] (b1) edge node {\scriptsize $1/3$} (b2);
    \path [->] (b2) edge node {\scriptsize $1/5$} (b3);
\end{scope}

\begin{scope}[every node/.style={ultra thick, align=center}]
    \node (c1) at (4.5+\xshift,\cypos) {1};
    \node (c2) at (6.0+\xshift,\cypos) {2};
    \node (c3) at (7.5+\xshift,\cypos) {3};
    \node (c4) at (9.0+\xshift,\cypos) {2 \cdotspace\large$\cdots$};
\end{scope}
\begin{scope}[every node/.style={above,},
              every edge/.style={draw=black,>=triangle 60}]
    \path [->] (c1) edge node {\scriptsize $1/4$} (c2);
    \path [->] (c2) edge node {\scriptsize $1/3$} (c3);
    \path [->] (c3) edge node {\scriptsize $1/5$} (c4);    
\end{scope}
\end{tikzpicture}
} 
\end{center}
This process is also known as a Markov chain, and in the setting we consider the
two models, Markov chains and random walks, are equivalent. (For the experts
reading, we are considering time-homogenous walks and Markov chains on finite
state-spaces.) Applications of random walks include:
\begin{compactitem}
\item Google's PageRank model~\cite{langville2006-book}. States are web-pages and transition probabilities correspond to hyperlinks between pages. 
\item Card shuffling. Markov chains can be used to show that shuffling a card deck seven times is sufficient, for example~\cite{Bayer-1992-shuffling,jonsson1997-cutoff}. In this case, the states correspond to permutations of a deck of cards and transitions are the result of a riffle shuffle.
\item Markov chain Monte Carlo and Metropolis Hastings~\cite{glynn2007-stochastic-simulation}. Here, the goal is to sample from complicated probability distributions that could model extrema of an function when used for optimization or the uniform distribution if the goal is a random instance of a complex mathematical objects such as matchings or Ising models. For optimization, states reflect the optimization variables and transitions are designed to stochastically ``optimize'' the function value; for random instance generation, states are the mathematical objects and transitions are usually simple local rearrangements. 
\end{compactitem}

One of the key properties of a Markov chain is its stationary distribution. This
models where we expect the walk to be ``on average'' as the process runs to
infinity. (Again, for the experts, we are concerned with the Ces\`{a}ro limiting
distributions.) To compute and understand these stationary distributions, we
turn to matrices. For an $N$ state Markov chain, there is an $N \times N$ column
stochastic matrix $\mP$, where $\mP_{ij}$ is the probability of transitioning to
state $i$ from state $j$.  The matrix for the previous chain is:
\[ \mP = \bmat{ 1/2 & 0 & 3/5 \\
	            1/4 & 2/3 & 1/5 \\
	            1/4 & 1/3 & 1/5 }.\] 
A stationary distribution on the states is a vector $\vx \in \mathbb{R}^N$ satisfying
\begin{align}
\vx_{i} = \sum_{j}\mP_{ij}\vx_{j}, \quad \sum_{i} \vx_{i} = 1, \quad \vx_{i} \ge 0,\; 1 \le i \le N.
\end{align}
The existence and uniqueness of $\vx$, as well as efficient numerical algorithms
for computing $\vx$, are all
well-understood~\cite{kemeny1960finite,stewart1994introduction}.

A natural extension of Markov chains is to have the transitions depend on the
past few states, rather than just the last one.  These processes are called
\emph{higher-order Markov chains} and are much better at modeling
data in a variety of applications including
airport travel flows~\cite{rosvall2014memory},
e-mail communication~\cite{rosvall2014memory},
web browsing behavior~\cite{chierichetti2012web},
and network clustering~\cite{krzakala2013spectral}.
For example, here is the transition probability table for a second-order
Markov chain model on the state space $\{1, 2, 3\}$:
\begin{equation}\label{eq:second-order-example}
\mbox{
\begin{tabularx}{0.85\linewidth}{@{}l@{\quad}X@{\;}X@{\;}XX@{\;}X@{\;}XX@{\;}X@{\;}X@{}}
\toprule
Second last state & 1 & & & 2 & & & 3 &  \\
\cmidrule(r){2-4} 	
\cmidrule(l){5-7} 	
\cmidrule(l){8-10} 	
Last state & 1 & 2 & 3 & 1 & 2 & 3 & 1 & 2 & 3 \\
\midrule	
Prob.~next state is 1 & 0   & 0   & 0 & 1/4 & 0 & 0   & 1/4 & 0   & 3/4 \\
Prob.~next state is 2 & 3/5 & 2/3 & 0 & 1/2 & 0 & 1/2 & 0   & 1/2 & 0   \\
Prob.~next state is 3 & 2/5 & 1/3 & 1 & 1/4 & 1 & 1/2 & 3/4 & 1/2 & 1/4 \\	
\bottomrule
\end{tabularx}
}
\end{equation}
Thus, for instance, if the last two states were $2$ (last state) and $1$ (second last state), then the next state will be $2$ with probability $2/3$ and $3$ with probability $1/3$. 

Stationary distributions of higher-order Markov chains show which states appear
``on average'' as the process runs to infinity. To compute them, we turn to
hypermatrices. We can encode these probabilities into a transition hypermatrix
where $\pijk$ is the probability of transitioning to state $i$, given that
the last state is $j$ and the second last state was $k$.  So for this example,
we have $\cmP_{321} = 1/3$. In this case, $\sum_{i} \pijk = 1$ for all
$j$ and $k$, and we call these stochastic hypermatrices. (For a $m$-order Markov
chain, we will have an $m+1$-order stochastic hypermatrix). The stationary
distribution of a second-order Markov chain, it turns out, is simply computed by
converting the second-order Markov chain into a first-order Markov chain on a
larger state space given by \emph{pairs} of states. To see how this conversion
takes place, note that if the previous two states were $2$ and $1$ as in the
example above, then we view these as an ordered pair $(2,1)$. The next pair of
states will be either $(2,2)$ with probability $2/3$ and $(3,2)$ with
probability $1/3$. Thus, the sequence of states generated by a second-order
Markov chain can be extracted from the sequence of states of a first-order
Markov chain defined on \emph{pairs} of
states. Figure~\ref{fig:second-order-markov} shows a graphical illustration of
the first-order Markov chain that arises from our small example.

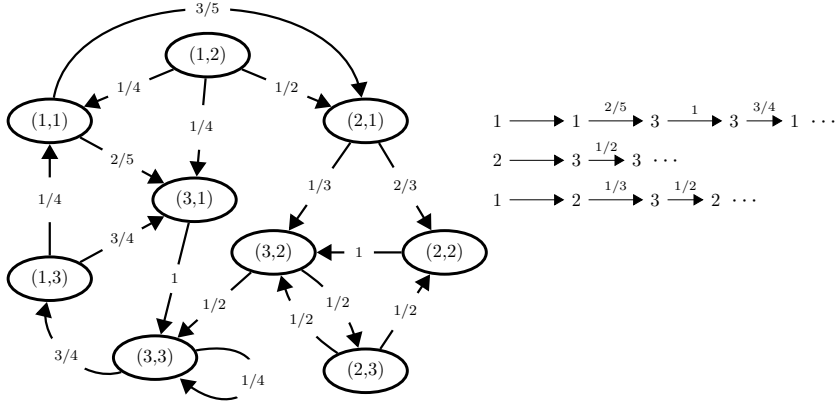
\begin{figure}
\begin{center}
  \scalebox{0.7}{
\begin{tikzpicture}
\begin{scope}[every node/.style={ultra thick, draw=black, ellipse, minimum width=45pt,
    align=center}]
    \node (11) at (5,15.5) {(1,1)};
    \node (21) at (11,15.5) {(2,1)};
    \node (31) at (7.75,14) {(3,1)};
    \node (12) at (8,16.75) {(1,2)};
    \node (22) at (12.5,13) {(2,2)};
    \node (32) at (9.25,13) {(3,2)};
    \node (13) at (5,12.5) {(1,3)};
    \node (23) at (11,10.75) {(2,3)};
    \node (33) at (7,11) {(3,3)};        
\end{scope}

\begin{scope}[every node/.style={fill=white,circle},
              every edge/.style={draw=black,very thick,>=triangle 60}]
    \path [->] (11) edge node {\footnotesize $2/5$} (31);
    \path [->] (11) edge[bend left=80] node {\footnotesize $3/5$} (21);
        
    \path [->] (21) edge node {\footnotesize $1/3$} (32);
    \path [->] (21) edge node {\footnotesize $2/3$} (22);
        
    \path [->] (31) edge node {\footnotesize $1$} (33);
    
    \path [->] (12) edge node {\footnotesize $1/4$} (11);
    \path [->] (12) edge node {\footnotesize $1/2$} (21);
    \path [->] (12) edge node {\footnotesize $1/4$} (31);    

    \path [->] (22) edge node {\footnotesize $1$} (32); 

    \path [->] (32) edge[bend left=20] node {\footnotesize $1/2$} (23);
    \path [->] (32) edge node {\footnotesize $1/2$} (33);    
    
    \path [->] (13) edge node {\footnotesize $1/4$} (11);
    \path [->] (13) edge node {\footnotesize $3/4$} (31);
    
    \path [->] (23) edge node {\footnotesize $1/2$} (22);
    \path [->] (23) edge[bend left=20] node {\footnotesize $1/2$} (32);    
    
    \path [->] (33) edge[bend left=60] node {\footnotesize $3/4$} (13);
    \path [->] (33) edge[out=10, in=-40, loop] node {\footnotesize $1/4$} (33);    
\end{scope}

\newcommand{\aypos}{15.5}
\newcommand{\bypos}{\aypos-0.75}
\newcommand{\cypos}{\bypos-0.75}
\newcommand{\xshift}{9}
\newcommand{\cdotspace}{\phantom{.}}

\begin{scope}[every node/.style={ultra thick, align=center}]
    \node (a1) at (4.5+\xshift,\aypos) {1};
    \node (a2) at (6.0+\xshift,\aypos) {1};
    \node (a3) at (7.5+\xshift,\aypos) {3};    
    \node (a4) at (9.0+\xshift,\aypos) {3};
    \node (a5) at (10.5+\xshift,\aypos) {1 \cdotspace\large$\cdots$};
\end{scope}
\begin{scope}[every node/.style={above,},
              every edge/.style={draw=black,>=triangle 60}]
    \path [->] (a1) edge node {} (a2);
    \path [->] (a2) edge node {\scriptsize $2/5$} (a3);
    \path [->] (a3) edge node {\scriptsize $1$} (a4);    
    \path [->] (a4) edge node {\scriptsize $3/4$} (a5);    
\end{scope}

\begin{scope}[every node/.style={ultra thick, align=center}]
    \node (b1) at (4.5+\xshift,\bypos) {2};
    \node (b2) at (6.0+\xshift,\bypos) {3};
    \node (b3) at (7.5+\xshift,\bypos) {3 \cdotspace\large$\cdots$};
\end{scope}
\begin{scope}[every node/.style={above,},
              every edge/.style={draw=black,>=triangle 60}]
    \path [->] (b1) edge node {} (b2);
    \path [->] (b2) edge node {\scriptsize $1/2$} (b3);
\end{scope}

\begin{scope}[every node/.style={ultra thick, align=center}]
    \node (c1) at (4.5+\xshift,\cypos) {1};
    \node (c2) at (6.0+\xshift,\cypos) {2};
    \node (c3) at (7.5+\xshift,\cypos) {3};
    \node (c4) at (9.0+\xshift,\cypos) {2 \cdotspace\large$\cdots$};
\end{scope}
\begin{scope}[every node/.style={above,},
              every edge/.style={draw=black,>=triangle 60}]
    \path [->] (c1) edge node {} (c2);
    \path [->] (c2) edge node {\scriptsize $1/3$} (c3);
    \path [->] (c3) edge node {\scriptsize $1/2$} (c4);    
\end{scope}
\end{tikzpicture}
  }
\end{center}
\vspace{-0.7cm}
\caption{%
The first-order Markov chain that results from converting the second-order
Markov chain in~\eqref{eq:second-order-example} into a first-order Markov chain
on pairs of states. Here, state $(i, j)$ means that the last state was $i$ and
the second-last state was $j$.  Because we are modeling a second-order Markov
chain, the only transitions from state $(i, j)$ are to state $(k, i)$ for some
$k$.
}
\label{fig:second-order-markov}
\end{figure}

The stationary distribution of the first-order chain is an $N\times N$ matrix $\mX$ where $\mX_{ij}$ is the stationary probability associated with the pair of states $(i,j)$. This matrix satisfies the stationary equations: 
\begin{align}
\mX_{ij} = \sum_{k} \pijk \mX_{jk},
\quad \sum_{i, j} \mX_{ij} = 1,
\quad \mX_{ij} \ge 0,\; 1 \le i, j \le N.\label{eqn:second_order_stationary}
\end{align}
(These equations can be further transformed into a matrix and a vector $\vx \in \mathbb{R}^{N^2}$ if desired, but for our exposition, this is unnecessary.)
The conditions when $\mX$ exists can be deduced from applying Perron-Frobenius theorem to this more complicated equation. Once the matrix $\mX$ is found,
the stationary distribution over states of the second-order chain is given by the row and column sums of $\mX$. 

Computing the stationary distribution in this manner requires
$\Theta(N^2)$ storage, regardless of any possible efficiency in storing and manipulating $\cmP$.  For modern problems on large datasets, this is infeasible.\footnote{We wish to mention that our goal is not the stationary distribution of the higher-order chain in a space efficient matter. For instance, a memory-friendly alternative is directly simulating the chain and computing an empirical estimate of the stationary distribution. Instead, our goal is to understand a recent \emph{algebraic} approximation proposed to these stationary distribution in terms of tensor eigenvectors.}

Recent work by~\citet{li2014limiting} provides a space-friendly alternative approximation through $z$ eigenvectors of the transition hypermatrix, as we now explain. \Citet{li2014limiting} consider a ``rank-one approximation'' to $\mX$, \emph{i.e}, $\mX_{ij} =
\vx_{i}\vx_{j}$ for some vector $\vx \in \mathbb{R}^{N}$.  
In this case, Equation~\eqref{eqn:second_order_stationary}
reduces to
\begin{align}
&\vx_{i} = \sum_{jk} \pijk\vx_{j}\vx_{k},
\quad \sum_{i} \vx_{i} = 1,
\quad \vx_{i} \ge 0,\; 1 \le i \le N.\label{eqn:ling_stationary}
\end{align}
Without the stochastic constraints on the vector entries, $\vx$ is called a $z$
eigenvector~\cite{qi2005eigenvalues} or an $l^2$
eigenvector~\cite{lim2005singular} of $\cmP$.  \citet{li2014limiting} and
\citet{gleich2015multilinear} analyze when a solution vector $\vx$ for
Equation~\ref{eqn:ling_stationary} exists and provide algorithms for computing
the vector.  These algorithms are guaranteed to converge to a unique solution
vector $\vx$ if $\cmP$ satisfies certain properties.  Because the entries of
$\vx$ sum to one and are nonnegative, they can be interpreted as a probability
vector.  However, this transformation was algebraic. 
We do not have a canonical process like the random
walk or second-order Markov chain connected to the vector $\vx$.
 
In this manuscript, we provide an underlying stochastic process, the \emph{spacey random walk},
where the limiting proportion of the time spent at each node---if this quantity
exists---is the $z$ eigenvector $\vx$ computed above. The process is an instance
of a \emph{vertex-reinforced random walk}~\cite{pemantle2007survey}.  The
process acts like a random walk but edge traversal probabilities depend on
previously visited nodes.  We will make this idea formal in the following
sections.

Vertex-reinforced random walks were introduced by \citet{coppersmith1987random} and \citet{pemantle1988random}, and
the spacey random walk is a specific type of a more general class of
vertex-reinforced random walks analyzed by~\citet{benaim1997vertex}.  A crucial
insight from Bena\"{i}m's analysis is the connection between the discrete
stochastic process and a continuous dynamical system.  Essentially, the limiting
probability of time spent at each node in the graph relates to the long-term
behavior of the dynamical system.  For our vertex-reinforced random walk, we can
significantly refine this analysis.  For example, in the special case of a
two-state discrete system, we show that the corresponding dynamical system for
our process always converges to a stable equilibrium
(Section~\ref{sec:two_state}).  
When our process has several states, we give
sufficient conditions under which the standard methods for dynamical systems will numerically
converge to a stationary distribution (Section~\ref{sec:numerical_convergence}).

We provide several applications of the spacey random walk process in
Section~\ref{sec:applications} and use the model to analyze a dataset of taxi
trajectories in Section~\ref{sec:data}.  Our work adds to the long
and continuing history of the application of tensor%
\footnote{These are coordinate representations of hypermatrices and are often
  incorrectly referred to as ``tensors.'' In particular, the transition
  hypermatrix of a higher-order Markov chain is a coordinate dependent and is
  not a tensor.  We refer to \citet{lim2013tensors} for formal definitions of
  these terms.}
or hypermatrix methods to data~
\cite{Harshman-1970-parafac,Comon:1994:ICA,SidiropoulosGB00,AllmanRhodes03:mb,SmilBG04,Sun-2006-tensor,anandkumar2014tensor}.
For example, eigenvectors of structured hypermatrix data are crucial to new
algorithms for parameter recovery in a variety of machine learning
applications~\cite{anandkumar2014tensor}.  In this paper, our hypermatrices come
from a particular structure---higher-order Markov chains---and we believe that
our new understanding of the eigenvectors of these hypermatrices will lead to
improved data analysis.

The software used to produce our figures and numerical results is publicly
available at \url{https://github.com/arbenson/spacey-random-walks}.

\section{The spacey random walk}\label{sec:process}

We now describe the stochastic process, which we call the \emph{spacey random walk}.  This process consists of a sequence of states $\xseq{0}, \xseq{1}, \xseq{2}, \ldots$, and uses the set of transition probabilities from a second-order Markov chain. In a second-order Markov chain, we use $\xseq{n}$ and $\xseq{n-1}$ to \emph{look-up} the appropriate column of transition data based on the last two
states of history.  
In the spacey random walker process on this data, once the process visits $\xseq{n}$, it
\emph{spaces out} and forgets its second last state (that is, the state
$\xseq{n-1}$). It then invents a new history state $\yseq{n}$ by randomly drawing a past state $\xseq{1}, \ldots, \xseq{n}$. Then it transitions to $\xseq{n+1}$ as
a second-order Markov chain as if its last two states were $\xseq{n}$ and
$\yseq{n}$.  We illustrate this process in Figure~\ref{fig:spacey}.

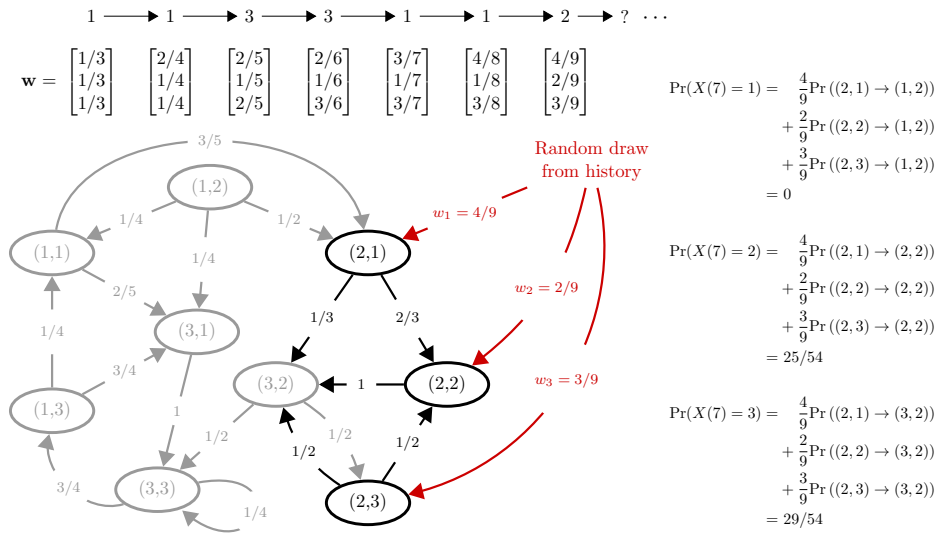
\begin{figure}
\centering
\scalebox{0.7}{
\begin{tikzpicture}

\newcommand{\aypos}{20}
\newcommand{\bypos}{\aypos-1.25}
\newcommand{\xshift}{1.25}
\newcommand{\cdotspace}{\phantom{.}}

\colorlet{MyGrey}{black!40!white}
\colorlet{TufteRed}{red!80!black}

\begin{scope}[every node/.style={ultra thick, align=center}]
    \node (a1)  at (4.5+\xshift,\aypos)  {1};
    \node (a2)  at (6.0+\xshift,\aypos)  {1};
    \node (a3)  at (7.5+\xshift,\aypos)  {3};    
    \node (a4)  at (9.0+\xshift,\aypos)  {3};
    \node (a5)  at (10.5+\xshift,\aypos) {1};
    \node (a6)  at (12.0+\xshift,\aypos) {1};
    \node (a7)  at (13.5+\xshift,\aypos) {2};
    \node (a8)  at (15.0+\xshift,\aypos) {? \cdotspace\large$\cdots$};
\end{scope}

\begin{scope}[every node/.style={above,},
              every edge/.style={draw=black,>=triangle 60}]
    \path [->] (a1) edge node {\scriptsize} (a2);
    \path [->] (a2) edge node {\scriptsize} (a3);
    \path [->] (a3) edge node {\scriptsize} (a4);
    \path [->] (a4) edge node {\scriptsize} (a5);
    \path [->] (a5) edge node {\scriptsize} (a6);
    \path [->] (a6) edge node {\scriptsize} (a7);
    \path [->] (a7) edge node {\scriptsize} (a8);
\end{scope}

\begin{scope}[every node/.style={ultra thick, align=center}]
    \node (w0)  at (3.5+\xshift,\bypos)  {$\ensuremath{\boldsymbol{\mathrm{w}}} =$};
    \node (w1)  at (4.5+\xshift,\bypos)  {$\begin{bmatrix} 1/3  \\ 1/3  \\ 1/3 \end{bmatrix}$};
    \node (w2)  at (6.0+\xshift,\bypos)  {$\begin{bmatrix} 2/4  \\ 1/4  \\ 1/4 \end{bmatrix}$};
    \node (w3)  at (7.5+\xshift,\bypos)  {$\begin{bmatrix} 2/5  \\ 1/5  \\ 2/5 \end{bmatrix}$};    
    \node (w4)  at (9.0+\xshift,\bypos)  {$\begin{bmatrix} 2/6  \\ 1/6  \\ 3/6 \end{bmatrix}$};
    \node (w5)  at (10.5+\xshift,\bypos) {$\begin{bmatrix} 3/7  \\ 1/7  \\ 3/7 \end{bmatrix}$};
    \node (w6)  at (12.0+\xshift,\bypos) {$\begin{bmatrix} 4/8  \\ 1/8  \\ 3/8 \end{bmatrix}$};
    \node (w7)  at (13.5+\xshift,\bypos) {$\begin{bmatrix} 4/9  \\ 2/9  \\ 3/9 \end{bmatrix}$};
\end{scope}

\begin{scope}[every node/.style={ultra thick, draw=black, ellipse, minimum width=45pt,
    align=center}]
    \node (21) at (11,15.5) {(2,1)};
    \node (22) at (12.5,13) {(2,2)};
    \node (23) at (11,10.75) {(2,3)};
\end{scope}

\begin{scope}[every node/.style={ultra thick, color=MyGrey, draw=MyGrey, ellipse, minimum width=45pt,
    align=center}]
    \node (11) at (5,15.5) {(1,1)};
    \node (31) at (7.75,14) {(3,1)};
    \node (12) at (8,16.75) {(1,2)};
    \node (32) at (9.25,13) {(3,2)};
    \node (13) at (5,12.5) {(1,3)};
    \node (33) at (7,11) {(3,3)};        
\end{scope}

\begin{scope}[every node/.style={fill=white,circle},
              every edge/.style={draw=black,very thick,>=triangle 60}]
    \path [->] (21) edge node {\footnotesize $1/3$} (32);
    \path [->] (21) edge node {\footnotesize $2/3$} (22);
    \path [->] (22) edge node {\footnotesize $1$} (32); 
    \path [->] (23) edge node {\footnotesize $1/2$} (22);
    \path [->] (23) edge[bend left=20] node {\footnotesize $1/2$} (32);    
\end{scope}

\begin{scope}[every node/.style={fill=white,circle},
              every edge/.style={color=MyGrey,draw=MyGrey,very thick,>=triangle 60}]
    \path [->] (11) edge node {\footnotesize $2/5$} (31);
    \path [->] (11) edge[bend left=80] node {\footnotesize $3/5$} (21);
    \path [->] (31) edge node {\footnotesize $1$} (33);
    
    \path [->] (12) edge node {\footnotesize $1/4$} (11);
    \path [->] (12) edge node {\footnotesize $1/2$} (21);
    \path [->] (12) edge node {\footnotesize $1/4$} (31);    
    \path [->] (32) edge[bend left=20] node {\footnotesize $1/2$} (23);
    \path [->] (32) edge node {\footnotesize $1/2$} (33);    

    \path [->] (13) edge node {\footnotesize $1/4$} (11);
    \path [->] (13) edge node {\footnotesize $3/4$} (31);
    \path [->] (33) edge[bend left=60] node {\footnotesize $3/4$} (13);
    \path [->] (33) edge[out=10, in=-40, loop] node {\footnotesize $1/4$} (33);    
\end{scope}

\begin{scope}[every node/.style={ultra thick, draw=white, text=TufteRed, minimum width=25pt,
    align=center}]
    \node (s) at (15.25,17.25) {Random draw\\from history};
\end{scope}
\begin{scope}[every node/.style={fill=white,circle},
              every edge/.style={color=TufteRed,draw=TufteRed,very thick,>=triangle 60}]
    \path [->] (s) edge node {\footnotesize $w_1 = 4/9$} (21);
    \path [->] (s) edge[bend left=20] node {\footnotesize $w_2 = 2/9$} (22);
    \path [->] (s) edge[bend left=45] node {\footnotesize $w_3 = 3/9$} (23);
\end{scope}

\node () at (19.25,14.75) {%
\scalebox{0.85}{%
\begin{minipage}{1cm}
\begin{align*}
\text{Pr}(X(7) = 1)
=&\phantom{{}+{}}\frac{4}{9}\text{Pr}\left((2, 1) \to (1, 2)\right) \\
&+\frac{2}{9}\text{Pr}\left((2, 2) \to (1, 2)\right) \\
&+\frac{3}{9}\text{Pr}\left((2, 3) \to (1, 2)\right) \\
=& \mbox{ } 0 \\
\\
\text{Pr}(X(7) = 2)
=&\phantom{{}+{}}\frac{4}{9}\text{Pr}\left((2, 1) \to (2, 2)\right) \\
&+ \frac{2}{9}\text{Pr}\left((2, 2) \to (2, 2)\right) \\
&+ \frac{3}{9}\text{Pr}\left((2, 3) \to (2, 2)\right) \\
=& \mbox{ } 25/54\\
\\
\text{Pr}(X(7) = 3)
=&\phantom{{}+{}}\frac{4}{9}\text{Pr}\left((2, 1) \to (3, 2)\right) \\
&+ \frac{2}{9}\text{Pr}\left((2, 2) \to (3, 2)\right) \\
&+ \frac{3}{9}\text{Pr}\left((2, 3) \to (3, 2)\right) \\
=& \mbox{ } 29/54
\end{align*}
\end{minipage}%
}};
\end{tikzpicture}
}
\vspace{-0.7cm}
\caption{The spacey random walk process uses a set of transition probabilities
	from a second-order Markov chain (which we represent here via the first-order reduced chain; see Figure~\ref{fig:second-order-markov}) and maintains an occupation vector $\vw$ as well. The vector $\vw$ keeps track of the fraction of time spent in each state, initialized
with one visit at each state.  In the next transition, the process chooses the past state
following this vector (red lines) and then makes a transition following the second-order
Markov chain (black lines). In this case, the last state is $2$ and the second last state will be generated from $\vw$, giving the transition probabilities to the next state $\xseq{7}$ on the right.}
\label{fig:spacey}
\end{figure}

We formalize this idea as follows.  Let $\pijk$ be the transition probabilities
of the second-order Markov chain with $N$ states such that
\[ \prob{\xseq{n+1} = i \given \xseq{n} = j, \xseq{n-1}=k} = \pijk. \] 
The probability law of the spacey random surfer is given by
\begin{align}
\prob{\yseq{n}=k \given \mathcal{F}_n} & = \frac{1}{n+N} (1 + \textstyle\sum_{s=1}^n \Indof{\xseq{s} = k})\label{eqn:history_draw} \\
\prob{\xseq{n+1} = i \given \xseq{n} = j, \yseq{n}=k} & = \pijk,\label{eqn:srw_step}
\end{align}
where $\mathcal{F}_n$ is the $\sigma$-field generated by the random variables
$\xseq{i}$, $i = 1, \ldots, n$ and $\xseq{0}$ is a provided starting state.
This system describes a random process with
reinforcement~\cite{pemantle2007survey} and more specifically, a new type of
generalized vertex-reinforced random walk (see
Section~\ref{sec:relationship_vrrw}).  For notational convenience, we let the $N
\times N^2$ column-stochastic matrix $\mR$ denote the flattening of $\cmP$ along
the first index:
\[
\mR := \left[ \begin{array}{c|c|c|c} 
\paak{1} & \paak{2} & \ldots & \paak{N}
\end{array} \right].
\]
(We will use the term \emph{flattening}, but we note that this operation is also referred to as an \emph{unfolding} and
sometimes denoted by $\mR = \cmP_{(1)}$~\cite[Chapter~12.1]{golub2012matrix}.)
We refer to $\paak{i}$ as the $i$th \emph{panel} of $\mR$ and note that each panel
is itself a transition matrix on the original state space.  Transitions of the spacey random walker then
correspond to: (1) selecting the panel from the random variable $Y(n)$ in
Equation~\ref{eqn:history_draw} and (2) following the transition probabilities
in the column of the panel corresponding to the last state.

\subsection{Relationship to vertex-reinforced random walks}
\label{sec:relationship_vrrw}

The spacey random walk is an instance of a (generalized) vertex-reinforced
random walk, a process introduced by \citet{benaim1997vertex}.
A vertex-reinforced random walk is a process that always picks the next
state based on a set of transition probabilities, but those transition
probabilities evolve as the process continues.  A simple example would
be a random walk on a graph where the edge traversals are proportional
to the amount of time spent at each vertex so far~\cite{diaconis1988recent}.
At the start of the process, each vertex is assigned a score of 1, and after
each transition to vertex $i$ from vertex $j$, the score of
vertex $i$ is incremented by 1.  At any point in the process, transitions
from vertex $j$ are made proportional to the score of vertex $j$'s
neighbors in the graph.  As discussed by \citet{diaconis1988recent}, this
process could model how someone navigates a new city, where
familiar locations are more likely to be traversed in the future.

Formally, the vertex-reinforced random walk is a stochastic process with states
$\xseq{i}$, $i = 0, 1, \ldots$ governed by
\begin{align}
&\xseq{0} = x(0) \\
&\vs_i(n) = 1 + \sum_{s=1}^{n}\Indof{\xseq{s} = i}, \quad  \occseq{n} = \frac{\vs(n)}{N + n} \\
&\prob{\xseq{n+1} = i \given \mathcal{F}_n} = \left[\occmat(\occseq{n})\right]_{i, \xseq{n}}
\end{align}
where again $\mathcal{F}_n$ is the $\sigma$-field generated by the random variables
$\xseq{i}$, $i = 1, \ldots, n$, $x(0)$ is the initial state, and
$\occmat(\occseq{n})$ is a $N \times N$ column-stochastic matrix given by a map
\begin{equation}
\occmat\colon \Delta_{N-1} \to \{\mP \in \mathbb{R}^{N \times N} \given \mP_{ij} \ge 0, \; \ve^T\mP = \ve^T \},\quad
\vx \mapsto \occmat(\vx),
\end{equation}
where $\Delta_{N-1}$ is the simplex of length $N$-probability vectors that are non-negative and sum to one. 
In other words, transitions at any point of the walk may depend
on the relative amount of time spent in each state up to that
point.  The vector $\occ$ is called the \emph{occupation vector} because it
represents the empirical distribution of time spent at each state.  However, note that the initial state $\xseq{0}$ does not contribute to the occupation vector.

A key result from \citet{benaim1997vertex} is the relationship between the
discrete vertex-reinforced random walk and the following dynamical system:
\begin{align}\label{eqn:dynamical_system}
\frac{d\vx}{dt} = \pi(\occmat(\vx)) - \vx,
\end{align}
where $\pi$ is the map that sends a transition matrix to its stationary
distribution.  Essentially, the possible limiting distributions of the occupation
vector $\occ$ are described by the long-term dynamics of the system in
Equation~\ref{eqn:dynamical_system} (see Section~\ref{sec:limiting_vrrw} for
details).  Put another way, \emph{convergence of the dynamical system to a fixed point}
is equivalent to \emph{convergence of the occupation vector} for the stochastic
process, and thus, the convergence of the dynamical system implies 
the existence of a stationary distribution.
In order for the dynamical system to make sense, $\occmat(\vx)$ must
have a unique stationary distribution.  We formalize this assumption for the
spacey random walker in Section~\ref{sec:propB}.

\begin{proposition}\label{prop:spacey_vrrw}
The spacey random walk is a vertex-reinforced random walk defined by the map
\[
\vx \mapsto \sum_{k=1}^{N} \paak{k}\vx_k = \mR \cdot (\vx \kron \mI).
\]
\end{proposition}
\begin{proof}
Given $\yseq{n} = k$, the walker transitions from $\xseq{n} = j$ to $\xseq{n +
  1} = i$ with probability $\pijk$.  Conditioning on $\yseq{n}$,
\begin{align}
& \prob{\xseq{n+1} = i \given \xseq{0}, \ldots, \xseq{n}, \occseq{n}} \nonumber \\
&\qquad = \sum_{k=1}^{N} \prob{\xseq{n+1} = i \given \xseq{n}, \yseq{n}=k}\prob{\yseq{n} = k \given \occseq{n}} \nonumber \\
&\qquad = \sum_{k=1}^N\cmP_{i,\xseq{n},k}\vw_k(n) = \left[\mR \cdot (\occseq{n} \kron \mI)\right]_{i, \xseq{n}}. \nonumber
\end{align}
Hence, given $\occseq{n}$, transitions are based on the matrix
$\sum_{k=1}^{N}\paak{k}\occ_k(n)$.  Since $\occseq{n} \in \Delta_{N-1}$ and each
$\paak{k}$ is a transition matrix, this convex combination is also a transition
matrix.
\end{proof}

The set of stationary distributions for the spacey random walk are the
fixed points of the dynamical system in Equation~\ref{eqn:dynamical_system} 
using the correspondence to the spacey random walk map. These
are vectors $\vx \in \Delta_{N-1}$ for which
\begin{align}\label{eqn:stationary}
0 = \frac{d\vx}{dt} = \pi( \mR \cdot (\vx \kron I) ) - \vx 
\iff \mR \cdot (\vx \kron \vx) = \vx 
\iff \vx_i = \sum_{ij} \pijk \vx_j \vx_k 
\end{align}
In other words, fixed points of the dynamical system are the $z$ eigenvectors of $\cmP$ that
are nonnegative and sum to one, i.e., satisfy Equation~\ref{eqn:ling_stationary}.

\subsection{Intuition for stationary distributions}
\label{sec:intuition}
We now derive some intuition for why Equation~\ref{eqn:ling_stationary} must be
satisfied for any stationary distribution of a spacey random walk without
concerning ourselves with the formal limits and precise arguments, which are
provided in Section~\ref{sec:limiting_vrrw}.  Let $\occseq{n}$ be the occupation
vector at step $n$.  Consider the behavior of the spacey random process at some
time $n \gg 1$ and some time $n+L$ where $n \gg L \gg 1$. The idea is to
approximate what will happen if we ran the process for an extremely long time
and then look at what changed at some large distance in the future. Since $L \gg
1$ but $L \ll n$, the vector $\occseq{L+n} \approx \occseq{n}$, and thus, the
spacey random walker $\{\xseq{n}\}$ approximates a Markov chain with transition
matrix:
\[ \mM(\occseq{n}) = \mR \cdot (\occseq{n} \kron \mI) = \sum_{k} \pijk \occ_k(n). \]

Suppose that $\mM(\occseq{n})$ has a unique stationary distribution $\vx(n)$
satisfying $\mM(\occseq{n}) \vx(n) = \vx(n)$.  Then, if the process $\{ \xseq{n}
\}$ has a limiting distribution, we must have $\vx(n) = \occseq{n + L}$,
otherwise, the distribution $\vx(n)$ will cause $\occseq{n + L}$ to
change. Thus, the limiting distribution $\vx$ heuristically satisfies:
\begin{align}
\vx = \mM(\vx) \vx = \mR \cdot (\vx \kron \mI) \vx = \mR \cdot (\vx \kron \vx).\label{eqn:limit}
\end{align}
Based on this heuristic argument, then, we expect stationary distributions of
spacey random walks to satisfy the system of polynomial equations
\begin{align}\label{eqn:limit_poly}
\vx_i = \sum_{1 \le j, k \le N} \pijk \vx_j \vx_k.
 \end{align}
In other words, $\vx$ is a $z$ eigenvector of $\cmP$.  Furthermore, $\vx$ is the
Perron vector of $\mM(\occseq{n})$, so it satisfies
Equation~\ref{eqn:ling_stationary}.

Pemantle further develops these heuristics by considering the change to
$\occseq{n}$ induced by $\vx(n)$ in a continuous time
limit~\cite{pemantle1992vertex}.  To do this, note that, for the case $n \gg L
\gg 1$,
\[ \occseq{n + L} \approx \frac{n \occseq{n} + L \vx(n)}{n+L} = \occseq{n}
+ \frac{L}{n+L} (\vx(n) - \occseq{n}). \]
Thus, in a continuous time limit $L \to 0$ we have: 
\[ \frac{d \occseq{n}}{dL} \approx
 \lim_{L\to0} \frac{\occseq{n+L} - \occseq{n}}{L} = \frac{1}{n} (\vx(n) -
 \occseq{n}). \] Again, we arise at the condition that, if this process
 converges, it must converge to a point where $\vx(n) = \occseq{n}$.

\subsection{Generalizations}
\label{sec:generalization}
All of our notions generalize to higher-order Markov chains beyond
second-order.  Consider an order-$m$ hypermatrix $\cmP$ representing an
$(m-1)$th-order Markov chain.  The spacey random walk corresponds to the
following process:
\begin{enumerate}
\item The walker is at node $i$, spaces out, and forgets the last $m - 2$ states.
         It chooses the last $m - 2$ states $j_1, \ldots, j_{m-2}$ at random based on its history
         (Each state is drawn i.i.d.~from the occupation vector $\occ$).
\item The walker then transitions to node $i$ with probability $\cmP_{i j_1 \cdots j_{m-2}}$,
         $1 \le i \le N$.
\end{enumerate}
Analogously to Proposition~\ref{prop:spacey_vrrw}, the corresponding
vertex-reinforced random walk map is
\[
\vx \mapsto \mR \cdot ( \fullkron{\vx}{m-2} \kron \mI ),
\]
where $\mR$ is the $N \times N^{m-1}$ column-stochastic flattening of $\cmP$
along the first index.

One natural generalization to the spacey random walker is the \emph{spacey
  random surfer}, following the model in the seminal paper
by~\citet{page1999pagerank}.  In this case, the spacey random surfer follows the
spacey random walker model with probability $\alpha$ and teleports to a random
state with probability $(1 - \alpha)$; formally,
\[
\prob{\xseq{n+1} = i \given \xseq{n} = j, \yseq{n}=k} = \alpha \pijk + (1 - \alpha)\prtel_i,
\]
where $\prtel$ is the (stochastic) teleportation vector.  We will use this model
to refine our analysis.  Note that the spacey random surfer model is an instance
of the spacey random walker model with transition probabilities $\alpha\pijk +
(1 - \alpha)\prtel_i$. This case was studied more extensively in~\citet{gleich2015multilinear}.

\subsection{Property B for spacey random walks}
\label{sec:propB}

Recall that in order to use the theory from \citet{benaim1997vertex} and the relationship between the stochastic process and the dynamical system, we must have that that $\occmat(\vx)$ has a unique stationary distribution. We formalize this requirement: 
\begin{definition}[{{Property B}}]\label{def:prop_b}
We say that a spacey random walk transition
hypermatrix $\cmP$ satisfies Property B if the corresponding vertex-reinforced random
walk matrix $\mM(\vw)$ has a unique Perron vector when $\vw$ is on the interior
of the probability simplex $\Delta_{N-1}$.
\end{definition}
This property is trivially satisfied if $\cmP$ is strictly positive. 
We can use standard properties of Markov chains to generalize the case when Property B holds.
\begin{theorem}\label{thm:prop_b_sat}
 A spacey random walk transition hypermatrix $\cmP$ satisfies Property B if and only if $\mM(\vw)$ has a single recurrent class for some strictly positive vector $\vw > 0$. 
\end{theorem}
\begin{proof} 
Property B requires a unique stationary distribution for each $\vw$ on the interior of the probability simplex. This is equivalent to requiring a single recurrent class in the matrix $\mM(\vw)$. 
Now, the property of having a single recurrent class is determined by the \emph{graph of the non-zero elements} of $\mM(\vw)$ alone and does not depend on their value. For any vector $\vw$ on the interior of the simplex, and also any positive vector $\vw$, the resulting graph structure of $\mM(\vw)$ is then the same. This graph structure is given by the \emph{union} of graphs formed by each panel of $\mR$ that results when $\cmP$ is flattened along the first index. Consequently, it suffices to test any strictly positive $\vw > 0$ and test the resulting graph for a single recurrent class. 
\end{proof}

\section{Applications}
\label{sec:applications}

We now propose several applications of the spacey random walk and show how some
existing models are instances of spacey random walks.  We remind the reader that
in all of these applications, the entries of the hypermatrix $\cmP$ comes from a higher-order Markov chain, but the spacey random walk describes the dynamics of the process.

\subsection{Population genetics}

In evolutionary biology, population genetics is the study of the dynamics of the
distribution of alleles (types).  Suppose we have a population with $n$ types.
We can describe the passing of type in a mating process by a hypermatrix of
probabilities:
\[
\prob{\text{child is type $i$} \given \text{parents are types $j$ and $k$}} = \pijk.
\]
The process exhibits a natural symmetry $\pijk = \cmP_{ikj}$ because the order
of the parent types does not matter.  The spacey random walk traces the lineages
of a type in the population with a random mating process:
\begin{enumerate}
\item A parent of type $j$ randomly chooses a mate of type $k$ from the
         population, whose types are distributed from the occupation vector $\occ$.
\item The parents create a child of type $i$ with probability $\pijk$.
\item The process repeats with the child becoming the parent.
\end{enumerate}

A stationary distribution $\vx$ for the type distribution satisfies the
polynomial equations $\vx_{i} = \sum_{k}\pijk\vx_{j}\vx_{k}$.  In population
genetics, this is known as the Hardy-Weinberg
equilibrium~\cite{hartl1997principles}.  Furthermore, the dynamics of the
spacey random walk describe how the distribution of types evolves over
time.  Our results offer new insights into how these population distributions
evolve as well as how one might learn $\cmP$ from data.


\subsection{Transportation}\label{sec:transportation}

We consider the process of taxis driving passengers around to a set of
locations.  Here, second-order information can provide significant information.
For example, it is common for passengers to make length-2 cycles: to and from
the airport, to and from dinner, etc.  Suppose we have a population of
passengers whose base (home) location is from a set $\{1, \ldots, N\}$ of
locations.  Then we can model a taxi's process as follows:
\begin{enumerate}
\item A passenger with base location $k$ is drawn at random.
\item The taxi picks up the passenger at location $j$.
\item The taxi drives the passenger to location $i$ with probability $\pijk$.
\end{enumerate}

We can model the distribution of base locations empirically, i.e., the
probability of a passenger having base location $k$ is simply relative to the
number of times that the taxi visits location $k$.  In this model, the
stochastic process of the taxi locations follows a spacey random walk.  We
explore this application further in Section~\ref{sec:taxis}.

\subsection{Ranking and clustering}

In data mining and information retrieval, a fundamental problem is ranking a set
of items by relevance, popularity, importance,
etc.~\cite{manning2008introduction}.  Recent work
by~\citet{gleich2015multilinear} extends the classical PageRank
method~\cite{page1999pagerank} to ``multilinear PageRank,'' where the stationary
distribution is the solution to the $z$ eigenvalue problem $\vx = \alpha\cmR(\vx
\kron \vx) + (1 - \alpha)\prtel$ for some teleportation vector $\prtel$.  This
is the spacey random surfer process discussed in Section~\ref{sec:generalization}.
As discussed in Section~\ref{sec:computation}, the multilinear PageRank vector
corresponds to the fraction of time spent at each node in the spacey random
surfer model.  In related work,~\citet{mei2010divrank} use the stationary
distribution of Pemantle's vertex-reinforced random walk to improve rankings.

Another fundamental problem in data mining is network clustering, i.e.,
partitioning the nodes of a graph into clusters of similar nodes. The
definition of ``similar" varies widely in the
literature~\cite{schaeffer2007graph,fortunato2010community}, but almost all of
the definitions involve first-order (edge-based) Markovian properties.  Drawing
on connections between random walks and clustering, \citet{benson2015tensor}
used the multilinear PageRank vector $\vx$ to partition the graph based on
\emph{motifs}, i.e., patterns on several nodes~\cite{alon2007network}.
In this application, the hypermatrix $\cmP$ encodes transitions on motifs.

\subsection{P\'olya urn processes}\label{sec:urn}

Our spacey random walk model is a generalization of a P\'olya urn process, which
we illustrate with a simple example.  Consider an urn with red and green balls.
At each step, we (1) draw a random ball from the urn, (2) put the randomly drawn
ball back in, and (3) put another ball of the same color into the urn.  Here, we
consider the sequence of states $\xseq{n}$ to be the color of the ball put in
the urn in step (3).  The transitions are summarized as follows.

\begin{center}
\begin{tabular}{l c c}
\toprule
Last ball selected & \multicolumn{2}{l}{Randomly drawn ball} \\
& Red & Green \\
\midrule
Red & Red & Green \\
Green & Red & Green \\
\bottomrule
\end{tabular}
\end{center}
In this urn model, the spacey random walk transition probabilities are
independent of the last state, i.e.,
\begin{equation}\label{eqn:urn_indep}
\prob{\xseq{n+1} = i \given \xseq{n} = j, \yseq{n}=k} = \prob{\xseq{n+1} = i \given \yseq{n}=k}.
\end{equation}
Nonetheless, this is still a spacey random walk where the transition hypermatrix is
given by the following flattening:
\begin{align}
 \mR = \left[ \begin{array}{cc|cc} 
1 & 1 & 0 & 0 \\
0 & 0 & 1 & 1
\end{array} \right].
\label{eqn:polya_example}
\end{align}
By Equation~\ref{eqn:urn_indep}, each panel (representing the randomly drawn
ball) has a single row of ones.

We can generalize to more exotic urn models.  Consider an urn with red and green
balls.  At each step, we (1) draw a sequence of balls $b_1, \dots, b_m$, (2)
put the balls back in, and (3) put a new ball of color $C(b_1, \dots, b_m) \in
\{\text{red}, \text{green}\}$ into the urn.  This process can be represented by
a $2 \times 2^m$ flattening $\mR$, where the panels are indexed by $(b_1,
\ldots, b_m)$:
 \[
 \mR_{b_1, \ldots, b_m} =
 \left[ \begin{array}{ll} 
\Indof{C(b_1, \ldots, b_m) = \text{red}} & \Indof{C(b_1, \ldots, b_m) = \text{red}}  \\
\Indof{C(b_1, \ldots, b_m) = \text{green}} & \Indof{C(b_1, \ldots, b_m) = \text{green}} \\
\end{array}
\right].
\]

Furthermore, the spacey random walk also describes urn models
with probabilistic choices for the color.  For example, suppose we
draw a sequence of balls $b_1, \dots b_m$ and add a new red ball into
the urn with probability $p(b_1, \ldots, b_m)$ and a green ball with
probability $1 - p(b_1, \ldots, b_m)$.  Then, the panels of $\mR$ would be
 \[
 \mR_{b_1, \ldots, b_m} =
 \left[ \begin{array}{ll} 
p(b_1, \ldots, b_m)      & p(b_1, \ldots, b_m) \\
1 - p(b_1, \ldots, b_m) & 1 - p(b_1, \ldots, b_m) \\
\end{array}
\right].
\]

Although these two-color urn processes can be quite complicated, our analysis in
Section~\ref{sec:two_states} shows that, apart from easily identifiable corner
cases, the dynamics of the system always converge to a stable equilibrium point.
In other words, the fraction of red balls in the urn converges.

\section{Dynamics with two states}
\label{sec:two_states}

We now completely characterize the spacey random walk for the simple case where
there are only two states.  This characterization covers the P\'olya urn process
described in Section~\ref{sec:urn}.

\subsection{The general two-state model}\label{sec:two_state}

In two-state models, the probability distribution over states is determined by
the probability of being in any one of the two states.  This greatly simplifies
the dynamics.  In fact, we show in Theorem~\ref{thm:stable_point_2x2} that in
all cases, the dynamical system describing the trajectory of the spacey random
walker converges to a stable stationary point.

Consider the case of an order-$m$ hypermatrix $\cmP$ with each dimension equal
to two. Let $\mR$ be the flattening of $\cmP$ along the first index.  We define the
map $\vz$ that sends a single probability to the probability simplex on two points:
\[
\vz\colon [0, 1] \to \Delta_1,\quad
x \mapsto \begin{bmatrix} x \\ 1 - x \end{bmatrix}.
\]
\begin{remark}
 Since the identity is the only $2 \times 2$ stochastic matrix
without a unique Perron vector, Property $B$ (Definition~\ref{def:prop_b}) can be reduced to
 \[
x \in (0, 1) \rightarrow \mR \cdot ( \fullkron{\vz(x)}{m-2} \kron \mI ) \neq \mI
\]
in the case of a two-state spacey random walk.
\end{remark}

We have a closed form for the map $\pi$ that sends a $2
\times 2$ stochastic matrix, with a unique Perron vector, to the first
coordinate of its Perron vector. (As noted in the remark,
a $2 \times 2$ stochastic matrix has
a unique Perron vector if it is not the identity, so $\pi$ applies to all $2
\times 2$ stochastic matrices that are not the identity matrix.)  Specifically,
\[
\pi\left(\begin{bmatrix} p & 1 - q \\ 1 - p & q \end{bmatrix} \right) = \vz\left(\frac{1 - q}{2 - p - q}\right).
\]
The corresponding dynamical system on the first coordinate is then
\begin{equation}\label{eqn:srw_dynamical_system}
\frac{dx}{dt} = \bigl[\pi\bigl( \mR \cdot ( \fullkron{\vz(x)}{m-2} \kron \mI ) \bigr)\bigr]_1 - x.
\end{equation}
When the dependence on $\mR$ is clear from context, we will write this dynamical
as $dx / dt = f(x)$ to reduce the notation in the statements of our theorems.

Figure~\ref{fig:srs_dynamics} shows the dynamics for the following flattening
of a fourth-order hypermatrix:
\[
\mR = \left[ \begin{array}{cc|cc|cc|cc} 
0.925 & 0.925 & 0.925  &  0.075  & 0.925  &  0.075 & 0.075 & 0.075 \\
0.075 & 0.075 &  0.075 &   0.925 & 0.075  & 0.925  & 0.925 & 0.925
\end{array} \right].
\]
The system has three equilibria points, i.e., there are three values of
$x$ such that $f(x) = 0$ on the interval $[0, 1]$.  Two of the points are stable
in the sense that small enough perturbations to the equilibrium value will
always result in the system returning to the same point.

\begin{figure}[t]
\centering
\includegraphics[scale=0.35]{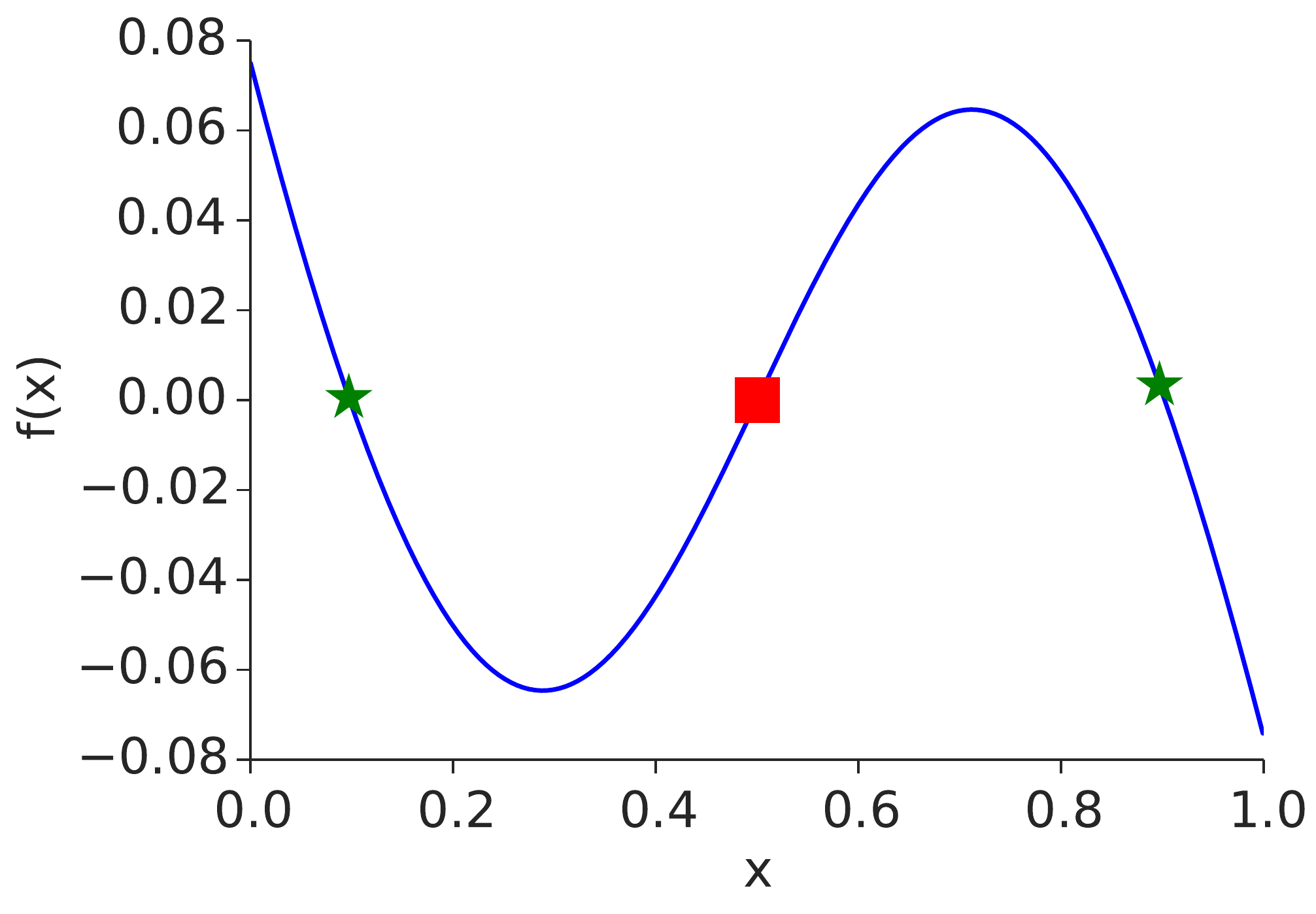}
\vspace{-0.4cm}
\caption{ The dynamics of the spacey random walk for a $2 \times 2 \times 2
  \times 2$ hypermatrix shows that the solution has three equilibrium points
  (marked where $f(x) = 0$) and two stable points (the green stars).  }
\label{fig:srs_dynamics}
\end{figure}

\subsection{Existence of stable equilibria}

The behavior in Figure~\ref{fig:srs_dynamics} shows that $f(0)$ is positive and
$f(1)$ is negative. This implies the existence of at least one equilibrium point
on the interior of the region (and in the figure, there are three). We now show
that these boundary conditions always hold.

\begin{lemma}\label{lem:boundary_conds}
Consider an order-$m$ hypermatrix $\cmP$ with $N=2$ that satisfies Property
B. The forcing function $f(x)$ for the dynamical system has the following
properties:
\begin{compactenum}[(i)]
\item $f(0) \ge 0$;
\item $f(1) \le 0$.
\end{compactenum}
\end{lemma}
\begin{proof}
Write the flattening of $\cmP$ as $\mR$ with panels $\mR_1, \ldots, \mR_{M}$ for
$M = 2^{m-2}$.  For part (i), when $x = 0$, $\vz(x) = \bmat{0 & 1}^T$ and
\[
\mR \cdot \bigl( \fullkron{\vz(x)}{m-2} \kron \mI \bigr) = \mR_{M} = \bmat{p & 1-q \\ 1-p & q },
\]
for some $p$ and $q$.  If $p$ and $q$ are both not equal to $1$, then $f(0) \ge
0$.  If $p = q = 1$, then we define $f(x)$ via its limit as $x \to 0$.  We have
\begin{align}
f(\epsilon) 
&= \left[
\pi\left(
(1 - \epsilon) \begin{bmatrix} 1 & 0 \\ 0 & 1 \end{bmatrix}+ \epsilon\begin{bmatrix} a & b \\ 1 - a & 1 - b \end{bmatrix}
\right)
\right]_1
 - \epsilon
 = \frac{b}{1 - a + b} - \epsilon. \nonumber
\end{align}
By Property B, $(a, b) \neq (1, 0)$ for any $\epsilon$, so $\lim_{\epsilon \to
  0} f(\epsilon) \ge 0$.

For part (ii), we have the same argument except, given that we are subtracting
$1$, the result is always non-positive.
\end{proof}

We now show that the dynamical system for the spacey random walker with two states
always has a stable equilibrium point.

\begin{theorem}\label{thm:stable_point_2x2}
Consider a hypermatrix $\cmP$ that satisfies property $B$ with $N=2$.  The
dynamical system given by Equation~\ref{eqn:dynamical_system} always converges
to a Lyapunov stable point starting from any $x \in (0,1)$.
\end{theorem}
\begin{proof}
Note that $f(x)$ is a ratio of two polynomials and that it has no singularities
on the interval $(0, 1)$ by Property $B$. Because $f(x)$ is a continuous
function on this interval, we immediately have the existence of at least one
equilibrium point by Lemma~\ref{lem:boundary_conds}.

If either $f(0) > 0$ or $f(1) < 0$, then at least one of the equilibria points
must be stable.  Consider the case when both $f(0) = 0$ and $f(1) = 0$.  If
$f(x)$ changes sign on the interior, then either $0$, $1$, or the sign change
must be a stable point. If $f(x)$ does not change sign on the interior, then
either $0$ or $1$ is a stable point depending on the sign of $f$ on the
interior.
\end{proof}

We note that it is possible for $f(x) \equiv 0$ on the interval $[0, 1]$.  For
example, this occurs with the deterministic P\'olya urn process described by the
flattened hypermatrix in Equation~\ref{eqn:polya_example}.  However, Lyapunov
stability only requires that, after perturbation from a stationary point, the
dynamics converge to a nearby stationary point and not necessarily to the
\emph{same} stationary point.  In this sense, any point on the interval is a
Lyapunov stable point when $f(x) \equiv 0$.

\subsection{The $2 \times 2 \times 2$ case}

We now investigate the dynamics of the spacey random walker for an order-$3$
hypermatrix $\cmP$ with flattening
\[
\mR = \left[ \begin{array}{cc|cc} 
a & b & c & d \\
1 - a & 1 - b & 1 - c & 1 - d
\end{array} \right]
\]
for arbitrary scalars $a$, $b$, $c$, $d \in [0, 1]$.
In this case, the dynamical system is simple enough to explore analytically.
Given any $x \in [0, 1]$,
\begin{align*}
\mR \cdot (\vz(x) \kron \mI)
&= x\bmat{a & b \\ 1 - a & 1 - b} + (1 - x)\bmat{c & d \\ 1 - c & 1 - d} \\
&= \bmat{c - x(c - a) & d - x(d - b) \\ 1 - c + x(c - a) & 1 - d + x(d - b)},
\end{align*}
and the corresponding dynamical system is
\begin{align}
\frac{dx}{dt} 
= 
\left[\pi\bigl( \mR \cdot (\vz(x) \kron \mI) \bigr)\right]_1 - x \nonumber
&= 
\frac{d - x(d - b)}{1 - c + d + x(c - a + b - d)} - x \nonumber \\
&=
\frac{d + (b + c -2d - 1)x + (a - c + d - b)x^2}{1 - c + d + x(c - a + b - d)}. \label{eqn:dynamics222}
\end{align}
Since the right-hand-side is independent of $t$, this differential equation is
separable:
\[
\int_{x_{0}}^{x}\frac{1 - c - d + (c - a + d - b)x}{d + (b + c - 2d - 1)x + (a  - c + d - b)x^2%
}dx=\int_{0}^{t}dt,
\]
where $x(0)=x_{0}$. Evaluating the integral,
\begin{align}
&\int\dfrac{\delta x+\epsilon}{\alpha x^{2}+\beta x+\gamma}\,dx \label{eqn:dynamics_2x2}\\
&=%
\begin{cases}
\frac{\delta}{2\alpha}\log\left\vert \alpha x^{2}+\beta x+\gamma\right\vert
+\frac{2\alpha\epsilon-\beta\delta}{\alpha\sqrt{4\alpha\gamma-\beta^{2}}}%
\tan^{-1}\left(  \frac{2\alpha x+\beta}{\sqrt{4\alpha\gamma-\beta^{2}}%
}\right)  +C & \text{for }4\alpha\gamma-\beta^{2}>0\\[15pt]%
\frac{\delta}{2\alpha}\log\left\vert \alpha x^{2}+\beta x+\gamma\right\vert
-\frac{(2\alpha\epsilon-\beta\delta)
\tanh^{-1}\left(  \frac{2\alpha x+\beta}{\sqrt{\beta^{2}-4\alpha\gamma}%
}\right)
}{\alpha\sqrt{\beta^{2}-4\alpha\gamma}%
}\,  +C & \text{for }4\alpha\gamma-\beta^{2}<0\\[15pt]%
\frac{\delta}{2\alpha}\log\left\vert \alpha x^{2}+\beta x+\gamma\right\vert
-\frac{2\alpha\epsilon-\beta\delta}{\alpha(2\alpha x+\beta)}+C &\text{for
}4\alpha\gamma-\beta^{2}=0
\end{cases} \nonumber
\end{align}
where
\[
\alpha=a - c + d - b,
\quad\beta=b + c - 2d - 1,
\quad\gamma=d,
\quad\delta=-c - a + d - b,
\quad
\epsilon=1-c-d,
\]
and $C$ is a constant determined by $x_0$.  Equation~\ref{eqn:dynamics222} also
gives us an explicit formula for the equilibria points.

\begin{proposition}\label{prop:stable_point}
The dynamics are at equilibrium ($\frac{dx}{dt} = 0$) if and only if
\[
 x = \left\{
     \begin{array}{ll}
      \dfrac{1 + 2d - b - c \pm \sqrt{(b + c - 2d - 1)^2 - 4(a - c + d - b)d}}{2(a - c + d - b)} & \textnormal{for } a + d \neq b + c\\
       \dfrac{d}{1 + 2d - b - c} & \textnormal{for } a + d = b + c,\\ &\phantom{for } b + c \neq 1 + 2d\\
     \end{array}
   \right.
\]
or for any $x \in [0, 1]$ if $a + d = b + c$ and $b + c = 1 + 2d$.  In this
latter case, the transitions have the form
\begin{equation}\label{eqn:corner_case}
\mR = \left[ \begin{array}{cc|cc} 
1 & b      & 1 - b & 0 \\
0 & 1 - b & b      & 1
\end{array} \right].
\end{equation}
\end{proposition}

\begin{proof}
The first case follows from the quadratic formula.  If $a + d = b + c$, then an
equilibrium point occurs if and only if $d + (b + c - 2d - 1)x = 0$.  If $b + c
\neq 1 + 2d$, then we get the second case.  Otherwise, $a = 1 + d$, which
implies that $a = 1$ and $d = 0$.  Subsequently, $b + c = 1$.
\end{proof}

Of course, only the roots $x \in [0, 1]$ matter for the general case $a + d \neq
b + c$.  Also, by Property B, we know that the value of $b$ in
Equation~\ref{eqn:corner_case} is not equal to $0$.  Finally, we note that for
the P\'olya urn process described in Section~\ref{sec:urn}, $a = b = 1$ and $c =
d = 0$, so the dynamics fall into the regime where any value of $x$ is an
equilibrium point, i.e., $dx / dt = 0$).

\section{Dynamics with many states: limiting distributions and computation}
\label{sec:computation}

Now that we have covered the full dynamics of the $2$-state case, we analyze the
case of several states.  Theorem~\ref{thm:spacey_conv} in
Section~\ref{sec:limiting_vrrw} formalizes the intuition from
Section~\ref{sec:intuition} that any stationary distribution $\vx$ of a spacey
random walk for a higher-order Markov chain $\cmP$ must be a $z$ eigenvector of
$\cmP$ (\emph{i.e.}, satisfy Equation~\ref{eqn:ling_stationary}).  Furthermore,
for any higher-order Markov chain, there always exists a stochastic $z$
eigenvector of the corresponding hypermatrix
$\cmP$~\cite[Theorem~2.2]{li2014limiting}.  However, this does not guarantee
that any spacey random walk will converge.

After, we study the relationship between numerical methods for computing $z$
eigenvectors and the convergence of spacey random walks and find two
conclusions.  First, the dynamics of the spacey random walk may converge even if
the commonly-used power method for computing the $z$ eigenvector does not
(Observation~\ref{obs:conv_power}).  On the other hand, in regimes where the
power method is guaranteed to converge, the dynamics of the spacey random surfer
will also converge (Theorem~\ref{thm:euler_conv}).  A complete characterization
between the computation of solutions to the algebraic sets of equations for the
$z$ eigenvector (Equation~\ref{eqn:ling_stationary}) and the convergence of the
continuous dynamical system is an open question for future research.

\subsection{Limiting behavior of the vertex-reinforced random walk}
\label{sec:limiting_vrrw}

We now summarize Bena\"{i}m's key result for relating the limiting distribution
of the occupation vector $\occ$ for a vertex-reinforced random walk
to the dynamical system.  Recall from Equation~\ref{eqn:dynamical_system} that
for an arbitrary vertex-reinforced random walk, the dynamical system is defined by
$d\vx / dt = \pi(\mM(\vx)) - \vx = f(\vx)$.  Since the function is smooth and $\Delta$ is
invariant under $f$, it generates a flow $\Phi\colon \mathbb{R}_+ \times \Delta \to \Delta$
where $\Phi(t, \vu)$ is the value at time $t$ of the initial value
problem with $\vx(0) = \vu$ and dynamics $f$.  A continuous function $X\colon \mathbb{R}_+ \to \Delta$ is
said to be an \emph{asymptotic pseudotrajectory}~\cite{benaim1996asymptotic} of
$\Phi$ if for any $L \in \mathbb{R}_+$,
\[
\lim_{t \to \infty} \| X(t + L) - \Phi(L, \vx(t)) \| = 0
\]
locally uniformly.  The following theorem, proved by~\citet{benaim1997vertex},
provides an asymptotic pseudotrajectory of the dynamical system in terms of the
occupation vector $\occ$ of a vertex-reinforced random walk.
\begin{theorem}[{{\cite{benaim1997vertex}}}]\label{thm:benaim}
Let $\tau_0 = 0$ and $\tau_n = \sum_{i=1}^n 1 / (i + 1)$.  Define the continuous
function $W$ by $W(\tau_n) = \occseq{n}$ and $W$ affine on $[\tau_n,
  \tau_{n+1}]$.  In other words, $W$ linearly interpolates the occupation vector
on decaying time intervals.  Then $W$ is almost surely an asymptotic
pseudotrajectory of $\Phi$.
\end{theorem}

Next, we put Bena\"{i}m's results in the context of spacey random walks.  Let
$\mR$ be the flattening of an order-$m$ hypermatrix $\cmP$ of dimension $n$
representing the transition probabilities of an $m$th-order Markov chain and let
$\vw$ be the (random) occupation vector of a spacey random walk following $\mR$.
\begin{theorem}\label{thm:spacey_conv}
Suppose that $\mR$ satisfies Property B.  Then we can define a flow
$\Phi_S\colon \mathbb{R}_+ \times \Delta_{N-1} \to \Delta_{N-1}$, where
$\Phi_S(t, \vu)$ is the solution of the initial value problem
\[
\frac{d\vx}{dt} = \pi\bigl(\mR \cdot \bigl( \fullkron{\vx}{m-2} \kron \mI \bigr)\bigr) - \vx, \quad \vx(0) = \vu.
\]
Then for any $L \in \mathbb{R}_+$, $\lim_{t \to \infty} \| W(t + L) - \Phi_S(L,
\vx(t)) \| = 0$ locally uniformly, where $W$ is defined as in
Theorem~\ref{thm:benaim}.  In particular, if the dynamical system of the spacey
random walk converges, then the occupation converges to a stochastic $z$
eigenvector $\vx$ of $\cmP$.
\end{theorem}

This result does not say anything about the existence of a limiting
distribution, the uniqueness of a limiting distribution, or computational
methods for computing a stationary vector.  We explore these issues in the
following sections.

\subsection{Power method}
We now turn to numerical solutions for finding a stationary distribution of the
spacey random walk process.  In terms of the tensor eigenvalue formulation, 
an attractive, practical method is a ``power method'' or ``fixed point'' 
iteration~\citet{li2014limiting,chu2014second,gleich2015multilinear}:
\begin{align*}
\vx(n + 1) = \mR \cdot (\vx(n) \kron \vx(n)),\; n = 0, 1, \ldots.
\end{align*}

Unfortunately, the power method does not always converge.  
We provide one example from~\citet{chu2014second}, given by the
flattening
\begin{align}
\mR = \left[ \begin{array}{cc|cc}
0 & 1 & 1 & 1 \\
1 & 0 & 0 & 0 \\
\end{array} \right]. \label{eqn:nonconv_power}
\end{align}
Here we prove directly that the power method diverges.  (Divergence is also a
consequence of the spectral properties of $\mR$~\cite{chu2014second}).  The
steady-state solution is given by $\vx_1 = (\sqrt{5} - 1)/2$ and $\vx_2 = (3 -
\sqrt{5}) / 2$.  Let $a(n)$ be the first component of $\vx(n)$ so that $\vx(n)
= \begin{bmatrix} a(n) & 1 - a(n) \end{bmatrix}^T$.  Then the iterates are given
by $a(n + 1) = 1 - a(n)^2$.  Let $a(n) = \vx_1 - \epsilon$.  Then
\[
a(n + 1) = 1 - (x_1 - \epsilon)^2 \approx 1 - \vx_1^2 + 2\epsilon\vx_1 = \vx_1 + 2\epsilon \vx_1.
\]
And the iterate gets further from the solution,
\[
\lvert a(n + 1) - \vx_1 \rvert = 2\epsilon\vx_1 > \epsilon = \lvert a(n) - \vx_1 \rvert.
\]
Next, note that $a(n + 2) = (1 - (1 - a(n))^2)^2 \ge a(n)$ for $a(n) \ge \vx_1$.
Furthermore, this inequality is strict when $a(n) > \vx_1$.  Thus, starting at
$a(n) = \vx_1 + \epsilon$ will not result in convergence.  We conclude that the
power method does not converge in this case.  However, since this is a two-state
system, the dynamical system for the corresponding hypermatrix must converge to
a stable point by Theorem~\ref{thm:stable_point_2x2}.  We summarize this result
in the following observation.

\begin{observation}\label{obs:conv_power}
Convergence of the dynamical system to a stable equilibrium does not imply
convergence of the power method to a tensor eigenvector.
\end{observation}

\subsection{Numerical convergence of the dynamical system for the spacey random surfer case}
\label{sec:numerical_convergence}
We will use the case of the spacey random surfer process (Section~\ref{sec:generalization})
in many of the  following examples because it has enough structure to illustrate
a definite stationary distribution. Recall that in this process,
the walker transitions to a random node with probability $1 - \alpha$ 
following the teleportation vector $\prtel$.
 
\begin{remark} For any $\alpha < 1$, the spacey random surfer
process satisfies Property B for any stochastic teleportation vector $\vv$.
\end{remark} 

This remark follows because for any occupation vector, the matrix $\mM(\vw(n))$ 
for the spacey random surfer corresponds to the convex combination of a column stochastic matrix and a rank-1 column stochastic matrix. The rank-1 component emerges from the teleportation due to the surfer. These convex combinations are equivalent to the Markov chains that occur in Google's PageRank models. They only have only a single recurrent class determined by the location of the teleportation. By way of intuition, just consider the places that could be visiting \emph{starting} from anywhere the teleportation vector is non-zero. This class recurs because there is always a finite probability of restarting. Thus, the matrix $\mM(\vw(n))$ always has a unique stationary distribution~\cite[Theorem~3.23]{berman1994nonnegative}.

\citet{gleich2015multilinear} design and analyze several computational methods
for the tensor eiegenvalue problem in Equation~\ref{eqn:limit}. 
We summarize their results on the power method for the spacey random surfer
model.

\begin{theorem}[{{\cite{gleich2015multilinear}}}]\label{thm:gleich_power_convergence}
Let $\cmP$ be an order-$m$ hypermatrix for a spacey random surfer model with
$\alpha < 1 / ( m - 1)$ and consider the power method initialized with the
teleportation vector, i.e., $\vx(0) = \prtel$.
\begin{enumerate}
\item There is a unique solution $\vx$ to Equation~\ref{eqn:limit}.
\item The iterate residuals are bounded by
$\| \vx(n) - \vx \|_{1} \le 2\left[\alpha(m-1)\right]^n$.
\end{enumerate}
\end{theorem}

These previous results reflect statements about the tensor eigenvector problem, but
not the stationary distribution of the spacey random surfer. In the course of the new
results below, we establish a sufficient condition for the
spacey random surfer process to have a stationary distribution. 

We analyze the forward Euler scheme for numerically solving the dynamical
system of the spacey random surfer.  Our analysis depends on the parameter
$\alpha$ in this model, and in fact, we get the same $\alpha$ dependence as in
Theorem~\ref{thm:gleich_power_convergence}.  The following result gives
conditions for when the forward Euler method will converge and also
when the stationary distribution exists.

\begin{theorem}\label{thm:euler_conv}
Suppose there is a fixed point $\vx$ of the dynamical system.  The time-stepping
scheme $\vx(n + 1) = hf(\vx(n)) + \vx(n)$ to numerically solve the dynamical
system in Eq.~\ref{eqn:dynamical_system} converges to $\vx$ when $\alpha < 1/(m
- 1)$ and $h \le (1 - \alpha) / (1 - (m-1)\alpha)$.  In particular, since the
forward Euler method converges as $h \to 0$, the spacey random surfer process
always converges to a unique stationary distribution when $\alpha < 1/(m-1)$.
\end{theorem}
\begin{proof}
Consider a transition probability hypermatrix $\cmP$ representing the
transitions of an $m$th-order Markov chain, and let $\mR$ be the one-mode
unfolding of $\cmP$. The fixed point $\vx$ satisfies
\[
\vx = \alpha \mR (\fullkron{\vx}{m-1}) + (1 - \alpha)\vv.
\]
Let $\vy(n)$ be the stationary distribution satisfying
\[
\vy(n) = \alpha\mR (\fullkron{\vx(n)}{m-2} \kron \mI)\vy(n) + (1 - \alpha)\prtel,
\]
and let $\| \cdot \|$ denote the $1$-norm. Subtracting $\vx$ from both sides,
\begin{align}
\| \vy(n) - \vx \|
&= \alpha \| \mR \| \cdot \| \fullkron{\vx(n)}{m-2} \kron \vy(n) - \fullkron{\vx}{m-1}\|  \nonumber \\
&\le \alpha\left((m-2)\|\vx(n) - \vx\| + \|\vy(n) - \vx\|\right). \nonumber
\end{align}
The inequality on the difference of Kronecker products follows from
\citet[Lemma~4.4]{gleich2015multilinear}.  This gives us a bound on the distance
between $\vy(n)$ and the solution $\vx$:
\begin{align}
\| \vy(n) - \vx \| \le \frac{(m-2)\alpha}{1 - \alpha}\|\vx(n) - \vx\|. \nonumber
\end{align}
Taking the next time step,
\begin{align}
\vx(n + 1) - \vx
&= hf(\vx(n)) + \vx(n) - \vx \nonumber \\
&= h(\vy(n) - \vx) + (1 - h)(\vx(n) - \vx). \nonumber
\end{align}
Finally, we bound the error on the distance between $\vx(n)$ and the fixed point
$\vx$ as a function of the previous time step:
\begin{align}
\|\vx(n+1) - \vx\| 
&\le h\frac{(m-2)\alpha}{1 - \alpha}\|\vx(n) - \vx\| + (1 - h)\|\vx(n) - \vx\| \\
&=\left(1 - \frac{h(1 - (m-1)\alpha)}{1 - \alpha}\right)\|\vx(n) - \vx\|. \label{eqn:contraction}
\end{align}

Now fix $T = hn$ and let $\vu(T)$ denote the continuous time solution to the initial value problem at time $T$.  Then $\| \vu(T) - \vx \| \le \| \vu(T) - \vx(n) \| + \| \vx(n) - \vx \|$.  Since the dynamical system is Lipschitz continuous and Euler's method converges, the first term goes to $0$ as $h \to 0$.  Simultaneously taking the limit as $n \to \infty$, the second term is bounded by $2e^{-T(1 - (m-1)\alpha) / (1 - \alpha)}$ through Equation~\ref{eqn:contraction}.  Thus, the dynamical system converges.

Uniqueness of the fixed point $\vx$ follows
from~\cite[Lemma~4.6]{gleich2015multilinear}.
\end{proof}

We note that the forward Euler scheme is just one numerical method for analyzing
the spacey random surfer process.  However, this theorem gives the same convergence dependence on $\alpha$ as
Theorem~\ref{thm:gleich_power_convergence} gives for convergence of the power method.

\section{Numerical experiments on trajectory data}
\label{sec:data}

We now switch from developing theory to using the spacey random walk model to
analyze data.  In particular, we will use the spacey random walk to model
trajectory data, i.e., sequences of states.  We first show how to find the
transition probabilities that maximize the likelihood of the spacey random walk
model and then use the maximum likelihood framework to train the model on
synthetically generated trajectories and a real-world dataset of New York City
taxi trajectories.

\subsection{Learning transition probabilities}\label{sec:learning}

Consider a spacey random walk model with a third-order transition probability
hypermatrix and a trajectory $\xseq{1}, \ldots, \xseq{Q}$ on $N$ states.  With
this data, we know the occupation vector $\vw(n)$ at the $n$th step of the
trajectory, but we do not know the true underlying transition probabilities.
For any given transition probability hypermatrix $\cmP$, the probability of a
single transition is
\[
\prob{\xseq{n + 1} = i \given \xseq{n} = j}
= \sum_{k=1}^{N}\prob{\yseq{n} = k}\pijk = \sum_{k=1}^{N}\vw_k(n)\pijk.
\]
We can use maximum likelihood to estimate $\cmP$.  The maximum likelihood
estimator is the minimizer of the following negative log-likelihood minimization
problem:
\begin{equation}\label{eqn:optimization}
\begin{aligned}
& \underset{\cmP}{\text{minimize}}
& & -\sum_{q = 2}^{Q} \log\left(\sum_{k=1}^{N}\vw_k(q - 1)\cmP_{\xseq{q} \xseq{q - 1} k}\right) \\
& \text{subject to}
& &    \sum_{i=1}^{N}\pijk = 1, \; 1 \le j, k \le N, \quad 0 \le \pijk \le 1, \; 1 \le i, j, k \le N.
\end{aligned}
\end{equation}
(Equation~\ref{eqn:optimization} represents negative log-likelihood minimization
for a single trajectory; several trajectories introduces an addition summation.)
Since $\vw_k(q)$ is just a function of the data (the $\xseq{n}$), the objective
function is the sum of negative logs of an affine function of the optimization
variables.  Therefore, the objective is a smooth convex function.  For our
experiments in this paper, we optimize this objective with a projected gradient
descent algorithm.  The projection step computes the minimal Euclidean distance
projection of the columns $\cmP_{\bullet jk}$ onto the simplex, which takes linear
time~\cite{duchi2008efficient}.

In the above formulation we assume that storage of $\cmP$ is not an issue, which
will be the case for our examples in the rest of this section.  In practice, one could
modify the optimization problem to learn a sparse or structured $\cmP$.

\subsection{Synthetic trajectories}

We first tested the maximum likelihood procedure on synthetically generated
trajectories that follow the spacey random walk stochastic process.  Our
synthetic data was derived from two sources.  First, we used two transition
probability hypermatrices of dimension $N = 4$ from
\citet{gleich2015multilinear}:
\begin{align}
&\mR_1 = \left[ \begin{array}{cccc|cccc|cccc|cccc}
0 & 0 & 0 & 0 & 0 & 0 & 0 & 0 & 0 & 0    & 0 & 0 & 1/2 & 0 & 0 & 1 \\
0 & 0 & 0 & 0 & 0 & 1 & 0 & 1 & 0 & 1/2 & 0 & 0 & 0    & 1 & 0 & 0 \\
0 & 0 & 0 & 0 & 0 & 0 & 1 & 0 & 0 & 1/2 & 1 & 0 & 0    & 0 & 0 & 0 \\
1 & 1 & 1 & 1 & 1 & 0 & 0 & 0 & 1 & 0    & 0 & 1 & 1/2 & 0 & 1 & 0 \\
\end{array} \right], \nonumber \\ 
&\mR_2 = \left[ \begin{array}{cccc|cccc|cccc|cccc}
0 & 0 & 0 & 0 & 0 & 0 & 0 & 0 & 0 & 0    & 0 & 1 & 1 & 0 & 1 & 0 \\
0 & 0 & 0 & 0 & 0 & 1 & 0 & 1 & 0 & 1/2 & 0 & 0 & 0 & 1 & 0 & 0 \\
0 & 0 & 0 & 0 & 0 & 0 & 1 & 0 & 0 & 1/2 & 1 & 0 & 0 & 0 & 0 & 0 \\
1 & 1 & 1 & 1 & 1 & 0 & 0 & 0 & 1 & 0    & 0 & 0 & 0 & 0 & 0 & 1 \\
\end{array} \right]. \nonumber 
\end{align}
Note that the transition graph induced by $\mR_{i} \cdot (\vx \kron \mI)$ is
strongly connected for any $\vx > 0$ with self-loops on every node, $i = 1, 2$.
Thus, by Theorem~\ref{thm:prop_b_sat}, both $\mR_1$ and $\mR_2$ satisfy Property
B.  Second, we generated random transition probability hypermatrices of the same
dimension, where each column of the flattened hypermatrix is draw uniformly at
random from the simplex $\Delta_3$.  We generated 20 of these hypermatrices.
Since the entries in these hypermatrices are all positive, they satisfy Property
B.

For each transition probability hypermatrix, we generated 100 synthetic
trajectories, each with 200 transitions.  We used 80 of the trajectories as
training data for our models and the remaining 20 trajectories as test data.  We
trained the following models:
\begin{itemize}
\item The spacey random walk model where the transition probability hypermatrix
  is estimated by solving the maximum likelihood optimization problem
  (Equation~\ref{eqn:optimization}) with projected gradient descent.
\item A ``zeroth-order" Markov chain where the probability of transitioning to
  a state is the empirical fraction of time spent at that
  state in the training data.  We interpret this model as an ``intelligent"
  random guessing baseline.
\item A first-order Markov chain where the transitions are estimated empirically
  from the training data:
\begin{equation*}\label{eqn:fomc_mle}
\mP_{ij} =
\frac{\text{$\#$(transitions from $j$ to $i$)}}
{\text{$\sum_{l} \#$(transitions from $j$ to $l$)}}.
\end{equation*}
It is well-known that these transition probabilities are the
maximum likelihood estimators for the transition probabiliities of a first-order Markov chain.
\item A second-order Markov chain where the transition probabilities are
  estimated empirically from the training data:
\begin{equation*}\label{eqn:somc_mle}
\pijk =
\frac{\text{$\#$(transitions from $(j, k)$ to $(i, j)$)}}
{\text{$\sum_{l} \#$(transitions from $(j, k)$ to $(l, j)$)}}.
\end{equation*}
As was the case for the first-order chain, these transition probabilities are
the maximum likelihood estimators for the transition probabilities.
\end{itemize}

\begin{table}[tb]
\centering
\caption{%
Root mean square error (RMSE) on three test data sets with zeroth-order (ZO),
first-order (FO), and second-order (SO) Markov chain (MC) models and spacey
random walk (SRW) models.  The true spacey random walk model was used to
generate the training and test data.  Thus, the RMSE for ``true SRW" represents
the best we can expect from any model.  For the ``Random" (Rand.) data, we list
the mean error plus or minus one standard deviation over 20 data sets.  For the
constructed transition probability hypermatrices given by $\mR_{1}$ and
$\mR_{2}$, the learned spacey random walk model out-performs the Markov chain
models.  For the randomly generated transitions, all the first- and second-order
Markov chain models and the learned SRW model predict as well as the true
generative model.
}
\vspace{-0.25cm}
\begin{tabular}{@{\hspace{-0.005cm}} c c c c c c}
\toprule
& ZOMC & FOMC & SOMC & learned SRW & true SRW \\
\midrule
$\mR_{1}$ & 0.690 & 0.465 & 0.457 & 0.434 & 0.434 \\
$\mR_{2}$ & 0.526 & 0.325 & 0.314 & 0.292 & 0.292 \\
Rand.    & 0.743$\pm$0.007 & 0.718$\pm$0.010 & 0.718$\pm$0.010 & 0.717$\pm$0.011 & 0.717$\pm$0.011 \\
\bottomrule
\end{tabular}
\label{tab:synthetic_results}
\end{table}
For every state visited in every trajectory of the test set, each model gives a
probability $p$ of transitioning to that state.  We computed the root mean
square error (RMSE), where the error contribution for each visited state is $1 -
p$.  The RMSEs are summarized in Table~\ref{tab:synthetic_results}.  We include
the test errors for the ``true" spacey random walk model, i.e., we use the known
synthetic parameters for prediction.  These errors represent the smallest errors
we could expect from any model.

For the transition probability hypermatrices $\mR_1$ and $\mR_2$, the spacey
random walk clearly out-performs the Markov chains and has performance on par
with the true underlying model.  At a high level, both $\mR_1$ and $\mR_2$ are
adversarial to Markov chain modeling because they have multiple fixed points
(vectors that satisfy Equation~\ref{eqn:ling_stationary}).  For instance, the
unit basis vectors $\ve_2$ and $\ve_3$ are each stationary points for both
hypermatrices.  In fact, through Matlab's symbolic toolbox, we found four
stationary points for each of the hypermatrices.  Although these fixed points
may not be limiting distributions of the spacey random walk, their presence is
evident in finite trajectories.  For example, there are several trajectories
with long sequences of state 2 or state 3 (corresponding to the $\ve_2$ and
$\ve_3$ stationary points) in addition to more heterogeneous trajectories.  The
spacey random walk is able to understand these sequences because it models the
occupation vector.

On the other hand, the random transition probability hypermatrices are much more
amenable to modeling by Markov chains.  Apart from the zeroth-order Markov
chain, each model has roughly the same performance and predicts as well as the
true model on the test data.  We verified through Matlab's symbolic toolbox that
each of the twenty random hypermatrices has exactly one stationary point.
Therefore, we do not see the same effects of multiple stationary points on
finite trajectories, as with $\mR_1$ and $\mR_2$.  Furthermore, if the spacey
random walk process converges, it can only converge to a single asymptotically
first-order Markovian process.

\subsection{Taxi trajectories}\label{sec:taxis}

\begin{figure}[t]
\centering
\includegraphics[height=6.4cm]{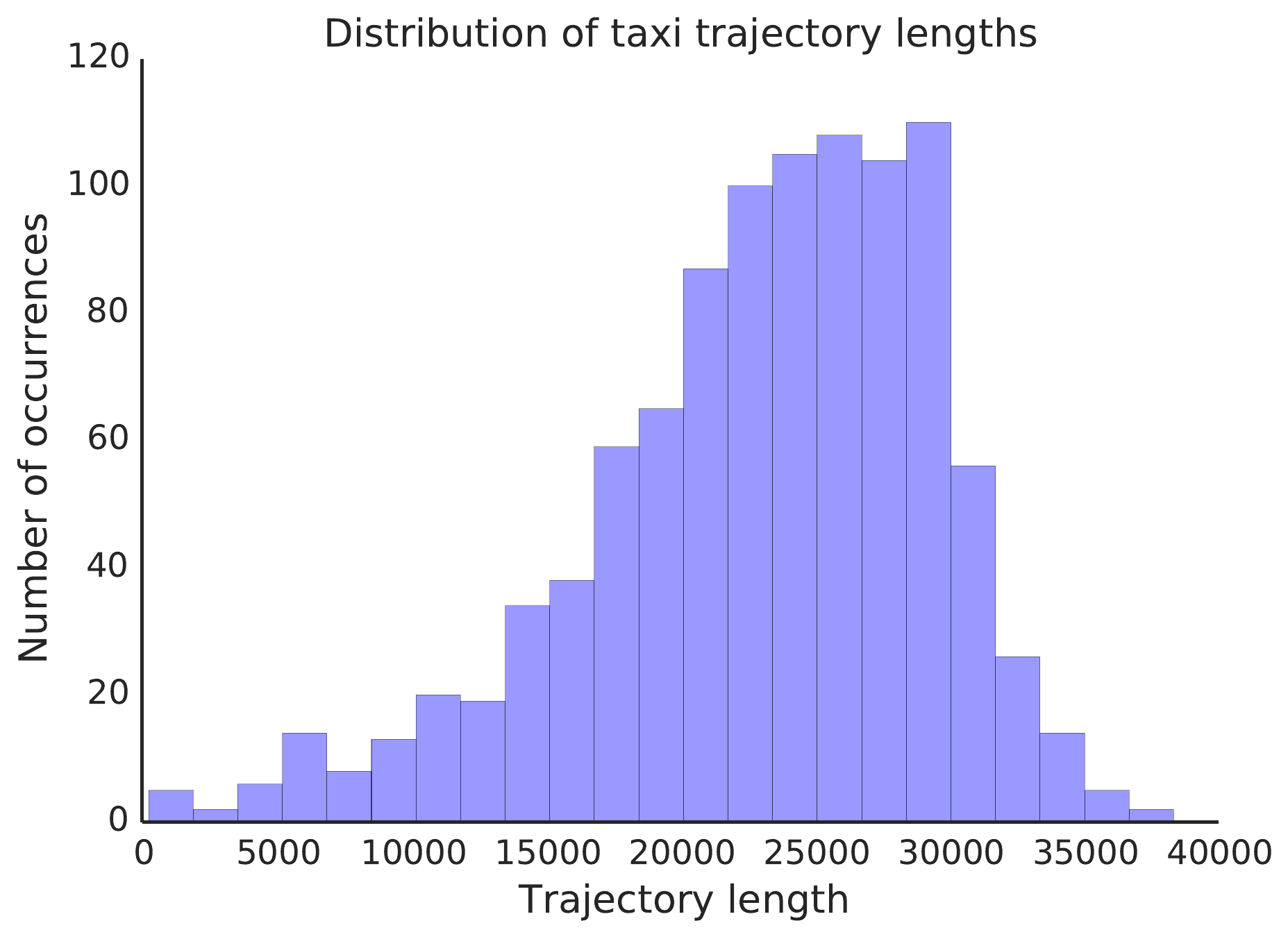}\quad
\includegraphics[height=8.0cm]{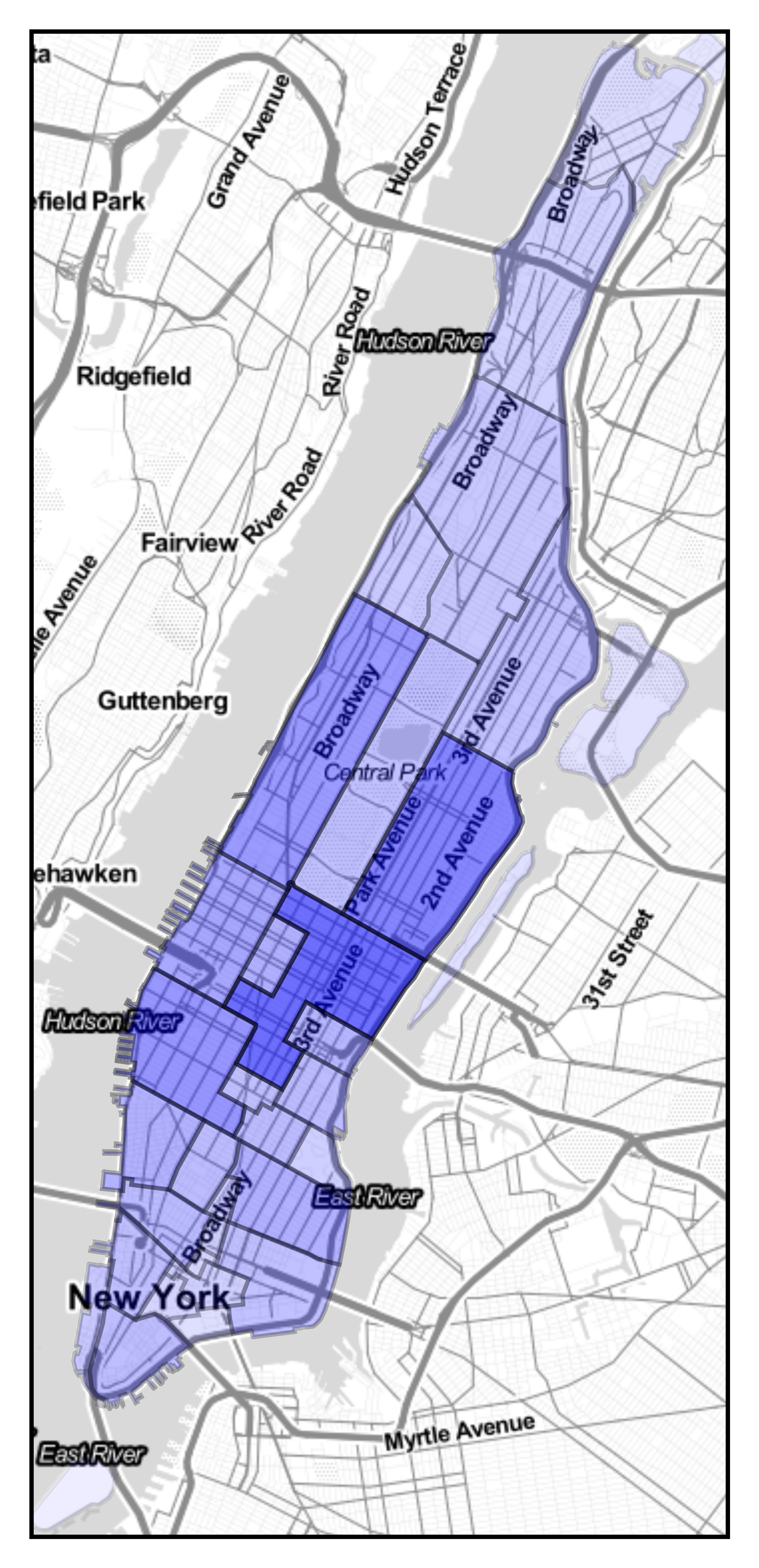}
\vspace{-0.6cm}
\caption{ Summary statistics of taxi dataset.  Left: distribution of trajectory
  lengths.  Right: empirical geographic frequency at the resolution of Manhattan
  neighborhoods (log scale).  Darker colors are more frequently visited by
  taxis.  }
\label{fig:taxi_data}
\end{figure}

Next, we modeled real-world transportation data with the spacey random walk
process.  We parsed taxi trajectories from a publicly available New York City
taxi dataset.%
\footnote{The original dataset is available at
  \url{http://chriswhong.com/open-data/foil_nyc_taxi/}.  The processed data used
  for the experiments in this paper are available with our codes.}
Our dataset consists of the sequence of locations for 1000 taxi drivers.
(Recall from Section~\ref{sec:transportation} how the spacey random walk models
this type of data.)  We consider locations at the granularity of neighborhoods
in Manhattan, of which there are 37 in our dataset.  One of these states is an
``outside Manhattan" state to capture locations outside of the borough.  If a
driver drops off a passenger in location $j$ and picks up the next passenger in
location $j$, we do not consider it a transition.  However, if a driver picks up
the next passenger in location $i \neq j$, we count the transition from location
$j$ to $i$.  Figure~\ref{fig:taxi_data} shows the lengths of the trajectories
and the empirical geographic distribution of the taxis.  Most trajectories
contain at least 20,000 transitions, and the shortest trajectory contains 243
transitions.  The geographic distribution shows that taxis frequent densely
populated regions such as the Upper East Side and destination hubs like Midtown.

We again compared the spacey random walk to the Markov chain models. We used 800
of the taxi trajectories for model training and 200 trajectories for model
testing.  Table~\ref{tab:taxi_results} lists the RMSEs on the test set for each
of the models.  We observe that the spacey random walk and second-order Markov
chain have the same predictive power, and the first-order Markov chain has worse
performance.  If the spacey random walk converges, we know from the results in
Section~\ref{sec:computation} that the model is asymptotically (first-order)
Markovian.  Thus, for this dataset, we can roughly evaluate the spacey random
walk as getting the same performance as a second-order model with asymptotically
first-order information.  We note that at first glance, the performance gains
over the first-order Markov chain appear only modest (0.011 improvement in
RMSE).  However, the results summarized in Table~\ref{tab:synthetic_results}
show that it is difficult to gain substantial RMSE improvements on a four-state
process that is actually generated from the spacey random walk process.  With the taxi
data, we have even more states (37) and the difference between the zeroth-order
Markov chain baseline to the first-order Markov chain RMSEs is only 0.076.  We
conclude that the spacey random walk model is not a perfect model for
this dataset, but it shows definite improvement over a first-order Markov chain
model.

\begin{table}[tb]
\centering
\caption{%
Root mean square errors on the test set for the taxi trajectory data.  The
spacey random walk and second-order Markov chain have the same performance and
are better predictors than a first-order Markov chain.  The zeroth-order Markov
chain corresponds to predicting based on the frequency spent at each state in
the training data.
}
\vspace{-0.25cm}
\begin{tabular}{c c c c c c}
\toprule
\multicolumn{1}{l}{Zeroth-order} & \multicolumn{1}{l}{First-order} & \multicolumn{1}{l}{Second-order} & \multicolumn{1}{l}{Learned Spacey}\\
\multicolumn{1}{l}{Markov chain} & \multicolumn{1}{l}{Markov Chain} & \multicolumn{1}{l}{Markov Chain} &  \multicolumn{1}{l}{Random Walk} \\
\midrule
0.922 & 0.846  & 0.835 & 0.835  \\
\bottomrule
\end{tabular}
\label{tab:taxi_results}
\end{table}

\section{Discussion}

This article analyzes a stochastic process related to eigenvector computations
on the class of hypermatrices representing higher-order Markov chains.  The
process, which we call the spacey random walk, provides a natural way to
think about tensor eigenvector problems which have traditionally been studied as
just sets of polynomial equations.  As is common with hypermatrix generalizations of
matrix problems, the spacey random walk is more difficult to analyze than the
standard random walk.  In particular, the spacey random walk is not
Markovian---instead, it is a specific type of generalized vertex-reinforced
random walk.  However, the intricacies of the spacey random walk make its
analysis an exciting challenge for the applied mathematics community.  In this
paper alone, we relied on tools tools from dynamical systems, numerical linear
algebra, optimization, and stochastic processes.

Following the work of \citet{benaim1997vertex}, we connected the limiting
distribution of the discrete spacey random walk to the limiting behavior of a
continuous dynamical system.  Through this framework, we fully characterized the
dynamics of the two-state spacey random walk and reached the positive result
that it always converges to a stable equilibrium point.  Analysis of the general
case is certainly more complicated and provides an interesting challenge for future
research.

One major open issue is understanding the \emph{numerical} convergence
properties of the spacey random walk, including algorithms for computing the
limiting distribution (if it exists).  For standard random walks, a simple power
method is sufficient for this analysis.  And generalizations of the power method
have been studied for computing the stationary distribution in
Equation~\ref{eqn:ling_stationary}%
~\cite{chu2014second,gleich2015multilinear,li2014limiting}.  However, we showed
that convergence of the dynamical system does \emph{not} imply convergence of
the power method for computing the stationary distribution of the spacey random
walk.  Nevertheless, the power method still often converges even when
convergence is not guaranteed by the theory~\cite{gleich2015multilinear}.  We
suspect that convergence of the power method is a sufficient condition for the
convergence of the dynamical system.

A final outstanding issue is determining \emph{when} the occupation vector of the
spacey random walk will converge.  We suspect that this problem is undecidable
as similar results have been shown for general dynamical
systems~\cite{buescu2011computability}.  One approach for this problem is to
show that the spacey random walker can simulate a Turing machine at which point
determining convergence is equivalent to the halting problem (see the work of
\citet{moore1990unpredictability} for a similar approach).

\section*{Acknowledgements}
We thank Tao Wu for observing subtleties in our use of the
term stable point in the analysis of the two-state case
when $f(x) \equiv 0$ and our convergence analysis of
the dynamical system in the limit of Euler's method
when the step size $h$ approaches $0$ in the limit.

\bibliographystyle{dgleich-bib}

\bibliography{refs}
 
\end{document}